\documentclass[12pt,reqno]{amsart}
\usepackage{amsmath,amsfonts,amssymb,amscd,amsthm,amsbsy,epsf}
\usepackage[dvips]{graphicx}
\usepackage[usenames]{xcolor}
\usepackage{comment}
\usepackage{float}
\usepackage{bm}

\newcommand\beq[1]{ \begin{equation}\label{#1} }
\newcommand{\eeq}{ \end{equation} }

\newcommand\beqa[1]{ \begin{eqnarray} \label{#1}}
\newtheorem{theorem}{Theorem}
\newtheorem{meta-thm}[theorem]{Meta-Theorem}
\newtheorem{lemma}[theorem]{Lemma}

\newtheorem{algorithm}[theorem]{Algorithm}
\newtheorem{remark}[theorem]{Remark}
\newtheorem{definition}[theorem]{Definition}

\newcommand{\eeqa}{ \end{eqnarray} }
\newcommand{\beqano}{ \begin{eqnarray*} }
\newcommand{\eeqano}{ \end{eqnarray*} }
\newcommand\equ[1]{{\rm (\ref{#1})}}

\def\u{\underline}

\def\ep{\varepsilon}

\def\P{{\mathcal P}}

\def\L{{\mathcal L}}
\def\M{{\mathcal M}}

\def\P{{\mathcal P}}

\def\integer{{\mathbb Z}}
\def\nat{{\mathbb N}}

\def\real{{\mathbb R}}
\def\torus{{\mathbb T}}

\title[Breakdown of tori in standard maps]
{ {Breakdown of rotational tori in 2D and 4D conservative and
dissipative standard maps}}

\author[A.P. Bustamante]{Adrian P.  Bustamante}
\address{
Department of Mathematics, University of Roma Tor Vergata, Via
della Ricerca Scientifica 1, 00133 Roma (Italy)}
\email{bustamante@mat.uniroma2.it}

\author[A. Celletti]{Alessandra Celletti}
\address{
Department of Mathematics, University of Roma Tor Vergata, Via
della Ricerca Scientifica 1, 00133 Roma (Italy)}
\email{celletti@mat.uniroma2.it}

\author[C. Lhotka]{Christoph Lhotka}
\address{
Department of Mathematics, University of Roma Tor Vergata, Via
della Ricerca Scientifica 1, 00133 Roma (Italy)}
\email{lhotka@mat.uniroma2.it}

\thanks{Corresponding author: \sl E-mail address: \rm celletti@mat.uniroma2.it}
\thanks{A.C. and C.L. acknowledge EU H2020 MSCA ETN Stardust-R Grant Agreement 813644.
A.B. and C.L. acknowledge the MIUR Excellence Department Project
awarded to the Department of Mathematics, University of Rome Tor
Vergata, CUP E83C23000330006, and the project MIUR-PRIN 20178CJA2B
``New Frontiers of Celestial Mechanics: theory and Applications''.
C.L. acknowledges GNFM/INdAM}

\begin{document}
\maketitle

\baselineskip=18pt              

\begin{abstract}
We study the breakdown of rotational invariant tori  {in 2D and 4D
standard maps} by implementing three different methods. First, we
analyze the domains of analyticity of a torus with given frequency
through the computation of the Lindstedt series expansions of the
embedding of the torus and the drift term. The Pad\'e and
log-Pad\'e approximants provide the shape of the analyticity
domains by plotting the poles of the polynomial at the denominator
of the Pad\'e approximants.  {Secondly}, we implement a Newton
method to construct the embedding of the torus; the breakdown
threshold is then  {estimated} by looking at the blow-up of the
Sobolev norms of the embedding. Finally,  {we implement an
extension of Greene method to get information on} the breakdown
threshold of an invariant torus with irrational frequency by
looking at the stability of the periodic orbits with periods
approximating the frequency of the torus.

We apply these methods to 2D and 4D standard maps. The 2D maps can
either be conservative (symplectic) or dissipative ( {more}
precisely, conformally symplectic, namely a dissipative map with
the geometric property to transform the symplectic form into a
multiple of itself). The conformally symplectic maps depend on a
dissipative parameter and a drift term, which is needed to get the
existence of invariant attractors. The 4D maps are obtained
coupling $(i)$ two symplectic standard maps, or $(ii)$ two
conformally symplectic standard maps, or $(iii)$ a symplectic and
a conformally symplectic standard map.

 {Concerning the results, Pad\'e and Newton methods perform
well and provide reliable and consistent results (although we
implemented Newton method only for symplectic and conformally
symplectic maps). Our implementation of the extension of Greene
method is inconclusive, since it is computationally expensive and
delicate, especially in 4D non-symplectic maps, also due to the
existence of Arnold tongues.}
\end{abstract}



\section{Introduction}\label{sec:model}

\subsection{The goal of the work}
As a motivation for studying  {2D and 4D conservative and
dissipative maps, we mention that in low-dimensional systems (e.g.
Hamiltonian systems with less or equal than 2 d.o.f.) invariant
rotational tori provide a strong stability property, since they
divide the phase space into invariant regions. In higher
dimensional systems, the confinement is no more valid and the
orbits can go around the invariant tori, a well-studied phenomenon
known as \sl Arnold diffusion \rm (\cite{Arnold64}). In Celestial
Mechanics, an example of low-dimensional system described by a
Hamiltonian with 2 d.o.f. is given by the planar, circular,
restricted three-body problem. Taking a surface of section on a
given energy level, one obtains a 2D map. Releasing the assumption
of circular orbits and allowing the orbits of the primaries to be
elliptic, one obtains that the motion takes place in a 6D phase
space}. Taking a surface of section on a given energy level, one
obtains a 4D map.

It is worth mentioning that dissipative effects in Celestial
Mechanics are found at different  {size scales}; we just quote
Poynting-Robertson drag on small particles, the effect of the
atmosphere on Earth's satellites, tidal torques on planetary
satellites. Such effects greatly contribute to shaping the
 {evolution} of objects in the solar system.

Given this physical motivation, the breakdown of invariant tori
 {(namely, tori whose dynamics is smoothly conjugated to a
rotation)} in conservative and dissipative 2D and 4D maps is the
core subject of the current work. Before proceeding, let us
specify what we mean by breakdown of invariant tori. The maps we
are going to study are nearly-integrable and depend on a
perturbing parameter, say $\varepsilon$, such that the map is
integrable when $\varepsilon=0$ and it is non-integrable when
$\varepsilon\not=0$. An invariant torus with a fixed irrational
frequency (precisely, Diophantine), say $\omega$, exists in the
integrable case. The torus gets deformed when $\varepsilon$
increases, until a critical value, say $\varepsilon_c$, is reached
at which  {point} the invariant torus breaks down. We will
implement different methods to compute $\varepsilon_c$.  {We
notice that in the dissipative case, one needs to introduce a
drift parameter (see \cite{Moser67}, \cite{BroerHS96},
\cite{CallejaC10}).}

\subsection{Maps and frequencies}
 {We will consider conservative and dissipative maps. In particular,}
we start from the 2D standard map introduced by Chirikov in
\cite{Chirikov79}, which is an area-preserving map,  {and hence}
symplectic; we study also its dissipative version, obtained adding
a dissipative parameter and a drift term. We will consider also 4D
maps by coupling two 2D symplectic maps, thus obtaining the 4D
symplectic map,  {known as the \sl Froeschl\'e map\rm}. Next, we
consider the coupling of a symplectic and a dissipative 2D map,
providing what we call a \sl mixed \rm 4D map, depending on a
dissipative parameter and a 1D drift term. Finally, we consider
the coupling of two 2D dissipative maps, which generate a
dissipative 4D map, depending on two dissipative parameters and
two drift terms.  {When the dissipative parameters are equal, the
map is said conformally symplectic (\cite{CallejaCdlL2013}), which
means that it is a dissipative map} with the property that it
transforms the symplectic form into a multiple of itself.

We will mainly consider the classical standard map with only one
harmonic, but we will also provide a few results for maps with two
harmonics.

In the 2D maps, the 1D invariant tori are labeled by a scalar
irrational frequency. KAM theory (\cite{Kolmogorov54},
\cite{Arnold63a}, \cite{Moser62})  {assumes that} the irrational
frequency satisfies a Diophantine condition. One can construct
rational approximants to the frequency by taking the truncations
of the continued fraction expansion of the irrational number. Each
rational approximant corresponds to a periodic orbit, that
approximates better the invariant torus as the period gets longer.
In the 4D maps, the 2D invariant tori are labeled by a frequency
vector with two components; in this case, the construction of the
rational approximants is more elaborate, see, e.g.,
\cite{KimOstlund}, \cite{Schweiger1990}, \cite{Tompaidis96},
\cite{Tompaidis96bis}. We will implement the so-called
Jacobi-Perron algorithm, which gives an explicit procedure to
construct rational approximants to the 2D frequency vector of the
4D symplectic, mixed, dissipative standard maps.

\subsection{Computing the breakdown threshold}
 {In this work, we implement three different methods (to which
we refer as Pad\'e approximants, Newton method, Greene method) to
determine the breakdown threshold of rotational invariant tori.}
Let us briefly detail such methods.

Invariant tori are represented by an embedding function and
 {one needs to compute also} a drift term in the dissipative
case. The embedding and the drift can be expanded in Lindstedt
series, in powers of the perturbing parameter. The terms of the
Lindstedt series can be obtained through recursive formulae as
solutions of suitable cohomological equations. We compute these
terms up to a given order and then we determine the Pad\'e
approximants of the truncated series,  {namely polynomials whose
quotients have the same Taylor expansions as the truncation} (see
\cite{BustamanteC1}, \cite{BustamanteC2} for computations
concerning the dissipative 2D standard map). The locations of the
zeros of the polynomial at the denominator provide  {an
approximation of} the analyticity domain of the invariant tori in
the complex parameter plane. Pioneer works on the computation of
Lindstedt series, analyticity domains and breakdown of tori are
provided by \cite{GreeneP81}, where the 2D standard map is
studied, and \cite{BolltM}, \cite{KM}, where the 4D semi-standard
map is analyzed. The shape and size of the analyticity domain
strongly depends on the nature of the map and the characteristics
of the frequency. We compute also the logarithmic Pad\'e
approximants (\cite{Lla-Tom-95}), which complement the results
obtained by computing the Pad\'e approximants. Finally, we
determine the radii of convergence of the Lindstedt series, which
give  {a lower bound} of the breakdown threshold of the invariant
tori.

The Newton method allows us to construct an invariant torus for
symplectic and conformally symplectic maps through a quadratically
convergent procedure.  {In conformally symplectic systems,}
starting from the invariance equation for the embedding and the
drift, the Newton method relies on the so-called \sl automatic
reducibility \rm (\cite{LlaveGJV05}, \cite{CallejaCdlL2013}),
according to which, in the vicinity of an invariant torus, there
exists  {an explicit} change of coordinates that transforms the
linearization of the invariance equation for the torus into a
constant coefficient equation. We implement an explicit algorithm,
borrowed from \cite{CallejaCdlL2013}, to implement a Newton method
to determine the embedding and the drift. Then, we compute the
corresponding Sobolev norms,  {whose blow-up provides} an estimate
of the breakdown threshold of a torus with fixed frequency. In
fact, as already noticed and implemented in \cite{CallejaC10},
\cite{CCGL20c} (see also \cite{CCGL20a}, \cite{CCGL20b}), if the
torus exists, the Sobolev norms  {change} with the perturbing
parameter;  {but} when approaching the critical threshold of the
perturbing parameter, the Sobolev norms blow up. This behavior
allows us to  {approximate} the critical value of the perturbing
parameter associated to the symplectic and conformally symplectic
2D and 4D maps.  {Rigorous results behind this method can be found
in \cite{CallejaL10}, \cite{CallejaCdlL2013}. We remark that for
conformally symplectic systems the tori persist as normally
hyperbolic invariant manifolds (\cite{CCFL14}), so that the
breakdown of the tori coincides with the breakdown of the
conjugacy to rotations or the breakdown of normal hyperbolicity.}

A technique to determine the critical threshold by looking at the
behavior of the approximating periodic orbits was developed
 {by J. Greene} in \cite{Greene1979}, which contains an
implementation for the symplectic 2D standard map (see
 {\cite{FalcoliniL92}, \cite{McKay92} for a partial
justification of Greene method} and \cite{FoxM13}, \cite{MeissFox}
for an application to volume-preserving maps). However,  {the
symplectic standard} map is very peculiar, since it can be written
as the product of two involutions; as a consequence, the periodic
orbits can be determined  {among} the fixed points of one of these
involutions. Unfortunately, this decomposition into involutions
does not happen in the dissipative case, a fact that complicates
enormously the determination of the periodic orbits, as already
noticed in \cite{CCGL20c},  {requiring more computational time
than Pad\'e method or Sobolev criterion.} Despite this
complication, we  {attempt} to determine periodic orbits also in
the dissipative case, at the expenses of a big computational
effort. Then, we implement  {an extension of Greene method, which
is based on the study of the linearization of the periodic orbits
with frequency close to the rotation number of the invariant
torus. Besides the works \cite{FalcoliniL92}, \cite{McKay92}
providing pioneer results to justify the symplectic case, a
partial justification of an extension of Greene method for
conformally symplectic and dissipative systems is given in
\cite{CCFL14}, where it is proved that if a KAM torus exists, then
one can get information on the spectrum of the periodic orbits
with nearby frequencies; the proof requires adjusting parameters
and provides information also on the width of the Arnold tongues.}

\subsection{Content of the work}
In Section~\ref{sec:circ} we introduce the definitions of
symplectic and conformally symplectic systems; in
Section~\ref{sec:SM} we give explicit expressions for the 2D and
4D maps that we are going to study; in
Section~\ref{sec:frequencies} we discuss the rational
approximation to the irrational frequencies; in
Section~\ref{sec:invariant} we give the definition of invariant
tori; we provide in Section~\ref{sec:Lindstedt} the Lindstedt
series expansions of the hull functions representing the tori  and
we compute the Pad\'e approximants to study the analyticity
domains of the tori; Newton method to construct the invariant tori
is presented in Section~\ref{sec:newton} together with an
application of Sobolev criterion;  {some results about} periodic
orbits approximating the tori are described in
Section~\ref{sec:periodic}.

\section{Symplectic and conformally symplectic
systems}\label{sec:circ}

We consider a discrete system $f$ defined on a symplectic manifold
$\M=B\times\torus^{ {d}}$, where $B\subseteq\real^{ {d}}$ is an
open, simply connected domain with smooth boundary and $ {d}\geq
1$ denotes the number of degrees of freedom of the system. We
assume that $\M$ is endowed with a symplectic form $\Omega$; then,
for any vectors $\u u,\u v\in\real^{ {d}}$, one has \beq{omegax}
\Omega_{\u x}(\u u,\u v)=(\u u, {J}\u v)\ , \eeq  {where we assume
that $J$, the matrix representing $\Omega$ at $\u x$, is a
constant matrix, as it will be the case for the models studied in
this work (see Section~\ref{sec:SM}).} Symplectic and conformally
symplectic maps have specific geometric properties as given by the
following definition.




\begin{definition}\label{def:CS}
A diffeomorphism $f:\M\rightarrow\M$ is conformally symplectic, if there exists a function $\lambda:\M\rightarrow\real$ such that
\beq{CS}
f^*\Omega=\lambda\Omega\ ,
\eeq
 {where $f^*$ denotes the pull-back of $f$. The diffeomorphism is symplectic
if $\lambda=1$.}
\end{definition}

It is known (see, e.g., \cite{CallejaCdlL2013}) that for $ {d}\geq
2$, any function $\lambda$ satisfying \equ{CS} must be constant.
In the following, we will always consider $\lambda$ constant.

The condition \equ{CS} is equivalent to
$$
\Omega_{f(\u x)}(Df(\u x)\u u,Df(\u x)\u v)=\lambda \Omega_{\u
x}(\u u,\u v)\ ,\qquad \forall \u u,\u v\in\real^{ {d}}\ .
$$
Recalling \equ{omegax}, we have
$$
(Df(\u x)\u u,\  {J} Df(\u x)\u v)=\lambda\ (\u u, {J}\u v)\ ,
$$
which must be valid for any $\u u$, $\u v$, thus yielding
$$
Df(\u x)^T\  {J}\ Df(\u x)=\lambda  {J}\ .
$$

As remarked in \cite{CallejaCdlL2013}, Definition \ref{def:CS} can be
generalized to the case in which the phase space can be written as
the product of $j\geq 2$ manifolds, say
$$
\M=\M_1\times...\times\M_j
$$
with associated symplectic forms $\Omega_1$, ..., $\Omega_j$, such
that $\Omega=\Omega_1\otimes ...\otimes \Omega_j$. Then, \equ{CS}
is replaced by the generalized condition \beq{Lij} f^*\Omega =
\lambda_1\Omega_1\otimes ...\otimes \lambda_j\Omega_j \eeq with
{$\lambda_1,...,\lambda_j$} constants. Motivated by this remark,
we introduce the following notion of systems with \sl mixed \rm
symplectic and conformally symplectic properties; for short, we
will refer to them as \sl mixed systems. \rm

\begin{definition}\label{def:mixed}
Let us consider a domain $\M$ which can be written as
$\M=\M_1\times...\times\M_j$ with $j\geq 2$ and let $\Omega_1$,
..., $\Omega_j$ be the symplectic forms on $\M_1$, ..., $\M_j$.
Let $f:\M\rightarrow\M$ be a diffeomorphism with conformal factors
$\lambda_1$, ..., $\lambda_j$ as in \equ{Lij}. We say that $f$
represents a mixed system, if there exists an index $1\leq k<j$,
such that \beq{SCS} \lambda_1,...,\lambda_k\not=1\ , \qquad
\lambda_{k+1}=...=\lambda_j=1\ . \eeq
\end{definition}

 {We will refer to  \sl  dissipative \rm systems, when not all $\lambda_1$, ..., $\lambda_d$
are equal for $d\geq 2$, and they all are different from 1.}

In the case of dissipative, conformally symplectic and mixed
systems, we will consider a family of maps $f_{\u \mu}$, where $\u
\mu\in\real^{ {d}}$ is called the \sl drift  {parameter}. \rm The
reason for introducing extra parameters is that in the symplectic
case, without the need of parameters, there exist invariant tori
with different frequencies, while in the conformally symplectic
case one needs to add a drift parameter vector to have the
existence of invariant tori  {of prescribed frequency}. When
considering a family of maps $f_{\u \mu}$, we say that $f_{\u
\mu}$ is conformally symplectic  {if it satisfies \equ{CS}.}
Similarly, we extend to the family $f_{\u \mu}$ the
Definition~\ref{def:mixed} of mixed systems, whenever \equ{SCS} is
satisfied. Examples of symplectic, mixed, conformally symplectic
maps are given in Section~\ref{sec:SM}.

\section{Standard maps}\label{sec:SM}
 {In this Section, we introduce 2D conservative
and dissipative standard maps, and 4D symplectic, mixed, dissipative, and
conformally symplectic standard maps.}

\subsection{Symplectic and conformally symplectic 2D standard maps}
The conformally symplectic 2D standard map is defined by the discrete set of
equations \beqa{DSM2}
y_{n+1}&=&\lambda y_n+\mu+\ep V(x_n)\nonumber\\
x_{n+1}&=&x_n+y_{n+1}\ , \eeqa with $y_n\in\real$, $x_n\in\torus$,
{$\varepsilon\geq 0$ is the perturbing parameter, $0<\lambda\leq
1$ is the conformal factor}, $\mu\in\real$ is the drift parameter,
and $V$ is a regular, periodic function. The determinant of the
Jacobian of \equ{DSM2} is equal to $\lambda$. Hence, when
$\lambda=1$ and $\mu=0$, one obtains the symplectic standard map
(\cite{Chirikov79}), while for $\lambda>0$ and $\mu\not=0$, one
obtains the conformally symplectic standard map.

\subsection{ {Symplectic, conformally symplectic, mixed and dissipative 4D standard maps}}
 {We consider the map} described by the following equations for $y_n,w_n\in\real$, $x_n,z_n\in\torus$:
 {
\beqa{SMS}
y_{n+1}&=&\lambda_1 y_n+\mu_1+\ep {{\partial W}\over {\partial x}}(x_n,z_n;\gamma)\nonumber\\
x_{n+1}&=&x_n+y_{n+1}\nonumber\\
w_{n+1}&=&\lambda_2 w_n+\mu_2+\ep {{\partial W}\over {\partial z}}(x_n,z_n;\gamma)\nonumber\\
z_{n+1}&=&z_n+w_{n+1}\ ,
\eeqa}
where $\ep\geq0$ is the perturbing
parameter, $\gamma\geq0$ is the coupling parameter,
 {$0<\lambda_1,\lambda_2\leq 1$ are the conformal factors, $\mu_1$, $\mu_2\in\real$
are the drift parameters, the function $W$ is periodic in} $x_n$,
$z_n$, with zero average in both arguments. The system \equ{SMS}
is obtained by coupling two 2D standard maps;  {its Jacobian
amounts to $\lambda_1\lambda_2$.}  {We will often consider the
following expression for the function $W$ entering in \equ{SMS}:}
\beq{W}  {W}(x,z;\gamma)=-\cos x-\cos z-\gamma \cos(x-z)\ . \eeq

 {
We remark that the map \equ{SMS} is:
\begin{description}
    \item[$(i)$] symplectic, if $\lambda_1=\lambda_2=1$, $\mu_1=\mu_2=0$;
    \item[$(ii)$] conformally symplectic, if $\lambda_1=\lambda_2>0$ and $\mu_1,\mu_2\in\real$;
    \item[$(iii)$] mixed, if $\lambda_1>0$, $\lambda_2=1$, $\mu_1\in\real$, $\mu_2=0$;
    \item[$(iv)$] dissipative (contractive), if $0<\lambda_1\not=\lambda_2<1$ and $\mu_1,\mu_2\in\real$.
\end{description}
}

\begin{remark}\label{rem:omega}
For  {practical purposes, we} mention that the frequency vector
$\u\omega=(\omega_1,\omega_2)$ associated to an orbit of \equ{SMS}
can be  {approximated by the quantities
$$
\omega_1={1\over {N_0}}\sum_{j=1}^{N_0} y_j\ , \qquad
\omega_2={1\over {N_0}}\sum_{j=1}^{N_0} w_j
$$
for $N_0\in\integer_+$ sufficiently large.}
\end{remark}

\section{Diophantine vectors and best approximations}\label{sec:frequencies}

As it is well known, the existence of invariant tori can be proved
by KAM theory, provided a condition on the frequency vector, say
$\u\omega\in\real^{ {d}}$, is satisfied. More precisely, {KAM
theory requires} that $\u\omega$ satisfies a Diophantine
condition, which is introduced as follows.

\begin{definition}
The frequency vector $\u\omega\in\real^{ {d}}$ is said to satisfy
the Diophantine condition, if \beq{DCD} |\u\omega\cdot\u
m_1+m_2|^{-1}\leq \nu^{-1}|\u m_1|^\tau\ ,\qquad \u
m_1\in\integer^{ {d}}\backslash\{\u 0\}\ , \ \ m_2\in\integer \eeq
for some constants $0<\nu\leq 1$, $\tau\geq 1$.
\end{definition}

For $\tau> {d}-1$, the set of Diophantine vectors has full
Lebesgue measure in $\real^{ {d}}$.  {For the 2D maps, the
frequency is a scalar number, while for the 4D map the frequency
is a 2D vector, say $\u\omega=(\omega_1,\omega_2)$. By Liouville
theorem (\cite{Khinchin}), algebraic numbers satisfy the
Diophantine condition with $d=1$. We will consider the golden
ratio $\varpi={{\sqrt{5}-1}\over 2}$, which is the inverse of the
solution  of the quadratic equation $\varpi^2-\varpi-1=0$. We also
introduce the spiral mean $s\simeq 1.32472$, which is the solution
of the cubic equation $s^3-s-1=0$, and the number $\tau\simeq
1.83929$ as the solution of the cubic polynomial
$\tau^3-\tau^2-\tau-1=0$.}

 {
For the 2D maps, we mainly consider the frequencies $\varpi$ and $s^{-1}$. For the 4D maps,
we select the following 2D vectors:
\beqa{vectors}
\u\omega_s&=&(s-1,s^{-1})\simeq (0.3247,0.7549)\ ,\nonumber\\
\u\omega_u&=&(s^{-1},s-1)\simeq (0.7549,0.3247)\ ,\nonumber\\
\u\omega_\tau&=&(\tau^{-1},\tau-1)\simeq (0.5437,0.8393)\nonumber\\
\u\omega_g&=&(\varpi,s^{-1})\simeq (0.6180,0.7549)\nonumber\\
\u\omega_a&=&(s^{-1},s^{-2})\simeq (0.7549,0.5698)\nonumber\\
\u\omega_c&=&(2(s-1)-{s^{-1}\over {24}},s-1)\simeq
(0.6180,0.3247)\ . \eeqa For the Diophantine property and rational
approximations of 2D frequencies, we refer to the specialized
literature (\cite{KimOstlund}, \cite{Khinchin},
\cite{Schweiger1990}). The vectors $\u\omega_s$, $\u\omega_u$,
$\u\omega_\tau$ (considered in \cite{MeissSander},
\cite{Tompaidis96}) are Diophantine, being in the cubic field
generated, respectively, by $s$ and $\tau$; as for $\u\omega_g$,
we have verified numerically that it satisfies the Diophantine
inequality \equ{DCD} up to $|\u m_1|+|\u m_2|\leq 10^6$. The
Diophantine property of $\u\omega_a$, $\u\omega_c$ has been
studied in \cite{CFL}. An algorithm due to Jacobi and Perron (see,
e.g., \cite{Schweiger1990}) to compute the rational approximants
to 2D frequency vectors is briefly recalled in
Appendix~\ref{app:JP}.}


\section{Invariant tori}\label{sec:invariant}

 {In this Section, we introduce the definition and parametric representation  of rotational KAM invariant tori
for mapping systems.}

\begin{definition}\label{def:K}
Let $f_{\u\mu}:\M\rightarrow\M$ be a  {family of conformally
symplectic} diffeomorphisms defined on the manifold
$\M=B\times\torus^{ {d}}$ with $B\subseteq \real^{ {d}}$ an open,
simply connected domain with smooth boundary. Let
$\u\omega\in\real^{ {d}}$ be the frequency vector, satisfying the
Diophantine condition \equ{DCD}. A KAM torus with frequency
$\u\omega$ is an invariant torus described by an embedding $\u
K:\M\rightarrow\torus^{ {d}}$ and a drift $\u\mu\in\real^{ {d}}$,
such that the following invariance equation is satisfied:
\beq{inv} \u f_{\u\mu}\circ \u K(\u \psi)=\u K(\u\psi+\u\omega)\
,\qquad \u\psi\in\torus^{ {d}}\ . \eeq
\end{definition}

 {Notice that, if there exists $\u K$ satisfying \equ{inv}, then $\u K_{\u\sigma}(\u\psi)=\u K(\u\psi+\u\sigma)$
also satisfies \equ{inv} (see \cite{CallejaCdlL2013}). We also notice that, in the case of symplectic diffeomorphisms, Definition~\ref{def:K} applies
without the need to introduce the drift term.}

 {
Let us make explicit the invariance equation \equ{inv} for the map \equ{SMS}. Using vector
notation, we introduce the quantities
$$
\u\xi=(x,z)\ ,\qquad \u\eta=(y,w)\ ,\qquad \u\lambda=(\lambda_1,\lambda_2)\ ,\qquad \u\mu=(\mu_1,\mu_2)\ ,\qquad
\u V=({{\partial W}\over{\partial x}},{{\partial W}\over{\partial z}})\ .
$$
Using the equations in \equ{SMS}, we obtain the relation
\beq{lagr}
\u\xi_{n+1}-(I_2+\Lambda)\u\xi_n+\Lambda \u\xi_{n-1}=\u\mu+\varepsilon \u V(\u\xi;\gamma)\ ,
\eeq
where $I_2$ is the $2\times 2$ identity matrix and $\Lambda$ is the diagonal $2\times 2$ matrix
with non-zero components $\lambda_1$, $\lambda_2$.}

 {
A parameterization of an invariant torus with frequency $\u\omega\in\real^2$ can be constructed by
introducing the hull function $\u P:\torus^2\rightarrow\torus^2$, so that
\beq{param}
\u\xi=\u\psi+\u P(\u\psi)
\eeq
with the property that the flow in the parametric coordinate is linear:
$$
\u\psi_{n+1}=\u\psi_n+2\pi \u\omega\ .
$$
From \equ{SMS}, we obtain also that the components conjugated to $\u\xi$ are given by
$$
\u\eta=\u\omega+\u P(\u\psi+2\pi\u\omega)-\u P(\u\psi)\ .
$$
From \equ{lagr} and \equ{param}, we obtain that $\u P$ satisfies the following equation:
\beqa{eqP}
\u P(\u\psi+2\pi\u\omega)&-&(I_2+\Lambda)\u P(\u\psi)+\Lambda\u P(\u\psi-2\pi\u\omega)+2\pi(I_2-\Lambda)\u \omega\nonumber\\
&-&\u\mu-\varepsilon \u V(\u\psi+\u P(\u\psi);\gamma)=\u 0\ .
\eeqa
As before, if $\u P$ satisfies \equ{eqP}, then
$\u P_{\u \alpha}(\u\psi):= \u P(\u\psi+\u\alpha)+\u\alpha$
will also be solution of \equ{eqP}. To settle this
under-determinacy, we can use the normalization \beq{norm}
\int_{\torus^2}\u P(\u\psi)\ d\u\psi=0\ . \eeq We note that this
is equivalent to the normalization used in \cite{CallejaCdlL2013}.
}
The solution of equation \equ{eqP} in the form of series expansions will be the core topic of Section~\ref{sec:Lindstedt} below.
We refer to \cite{CCL22} for the formulation of \equ{eqP} for the map \equ{DSM2}.

\section{Lindstedt series and analyticity domains for the invariant tori}\label{sec:Lindstedt}

In this Section, we provide the formulae for the computation of
the Lindstedt series expansions of the hull function $\u P$
introduced in \equ{param} as well as of the drift term
 {$\u\mu$, so that they solve \equ{eqP} in the sense of power
series.}

\subsection{Iterative computation of the Lindstedt series}\label{sec:iterative}

We start by expanding the function  {$\u P$, as well as the drift
$\u\mu$, in power series of $\ep$ and by retaining the terms up
to} a given order $N\in\integer_+$, thus obtaining the truncated
Lindstedt series: { \beq{expN} \u P^{(N)}(\u\psi)=\sum_{j=1}^N \u
P_j(\u\psi)\ep^j\ ,\qquad \u\mu^{(N)}=\sum_{j=1}^N \u\mu_j\ep^j\ ,
\eeq where $\u P^{(N)}$ satisfies the normalization \equ{norm}.}
Let us introduce the operator  {$E_{\u\mu}(\u P)$} as the left
hand side of \equ{eqP}. Then, we can say that  {$\u P^{(N)}$ is
such that
$$
\|E_{\u\mu^{(N)}}(\u P^{(N)})\|=O(|\ep|^{N+1})\ .
$$}
The terms  {$\u P_j$, $\u \mu_j$} in \equ{expN} can be obtained as
follows. Consider the formal expansion of the vector function
 {$\u V$ as \beq{V1V2} \u V(\u \psi+\u
P(\u\psi);\gamma)=\sum_{j=0}^\infty \u V_j(\u\psi;\gamma)\ep^j ;
\eeq} we will give in Section~\ref{subsec:expansions} an iterative
 {formula for} the coefficients $\u V_j$. Inserting
\eqref{expN} and \eqref{V1V2} into \eqref{eqP}, one obtains:
 {\beqano &&\sum_{j=0}^\infty  \u P_j(\u\psi+2\pi\u\omega)\ep^j
-(I_2+\Lambda) \sum_{j=0}^\infty \u P_j(\u\psi)\ep^j +\Lambda
\sum_{j=0}^\infty \u P_j(\u\psi-2\pi\u\omega)\ep^j\nonumber\\
&+&2\pi(I_2-\Lambda)\u\omega-\sum_{j=0}^\infty \u\mu_j\ep^j -\ep \sum_{j=0}^\infty \u V_j(\u\psi;\gamma)\ep^j=\u 0\ .
\eeqano}
From the previous relations, equating same powers of $\ep$, the following equations are obtained for $j=0$:
 {\beq{eq-zerorder}
\u P_0(\u\psi+2\pi\u\omega) - (I_2+\Lambda) \u P_0(\u\psi) +\Lambda \u P_0(\u\psi-2\pi\u\omega)
+2\pi(I_2-\Lambda)\u\omega - \u\mu_0 = \u 0\ ,
\eeq}
while for $j>0$ the equations become:
 {\beq{eq-korder}
\u P_j(\u\psi+2\pi\u\omega) - (I_2+\Lambda) \u P_j(\u\psi) +\Lambda \u P_j(\u\psi-2\pi\u\omega)
- \u\mu_j = \u V_{j-1}(\u\psi;\gamma)\ .
\eeq}

\begin{remark}
Notice that, the coefficients $\u V_{j-1}$ at the right hand side of \eqref{eq-korder} depend only on the previously computed coefficients
$\u V_{0}$, $\u V_{1}$, ..., $\u V_{j-2}$.
\end{remark}

At the order $j=0$, a solution of \eqref{eq-zerorder} is given by
$\u P_0=\u 0$ and $\u\mu=\u\mu_0$ with components
$$
\mu_{0,1} = 2\pi(1-\lambda_1)\omega_1\ ,\qquad \mu_{0,2} = 2\pi(1 - \lambda_2)\omega_2\ .
$$
For $j\geq 1$,  {the two components of the} cohomology equations
\eqref{eq-korder} are of the following form
 {\beq{cohom}
\L_{\lambda,{\u\omega}} g({\u \psi})-\mu=\eta({\u \psi};\gamma)\ ,
\eeq
where the operator $\L_{\lambda,{\u\omega}}$ is defined as
$$
\L_{\lambda,{\u\omega}} g({\u \psi}):= g({\u \psi} + 2\pi{\u \omega}) - (1+\lambda)g({\u \psi}) + \lambda g({\u \psi} - 2\pi{\u \omega})\ ;
$$}
the unknowns in \equ{cohom} are the function $g:\torus^2
\longrightarrow  {\real}$ and the scalar $\mu \in  {\real}$,
whereas the function $\eta:\torus^2 \longrightarrow \real$ and the
parameters $\gamma, \lambda\in \real$ are given.  {The solution of
\equ{cohom} is obtaned through the following lemma.}

\begin{lemma}\label{lemma:sol-cohom-eq}
Given  {$\lambda\in \real$}, let ${\u \omega}\in\real^2$ be a
Diophantine vector and $\eta :\torus^2 \longrightarrow  {\real} $
be an analytic  function. Then, equation \equ{cohom}  {has a
unique solution $(g,\mu)$ satisfying the normalization
\beq{norm-g} \int_{\torus^2} g(\u\psi)\ d\u\psi=0\ , \eeq given
by} \beqa{sol-cohom-lemma}
\mu &=& -\hat \eta_{\u 0}(\gamma)\nonumber \\
g({\u\psi}) &=& \sum_{\u k \in \integer^2 \backslash \{\u 0\} }\frac{\hat \eta_{\u k}(\gamma)}{e^{2\pi i
{\u \omega}\cdot{\u k}} - (1+\lambda) + \lambda e^{-2\pi i {\u \omega}\cdot{\u k} }}
e^{ i {\u k} \cdot {\u\psi}}\ ,\qquad \u k\not=\u 0\ ,
\eeqa
where $\hat \eta_{\u k}(\gamma) $ are the Fourier coefficients of $\eta$.
\end{lemma}

\begin{proof}
Equation \eqref{cohom} can be solved in Fourier space. Consider the Fourier expansions
\beq{f-exp-g}
g({\u\psi}) = \sum_{{\u k}\in \integer^2} \hat g_{\u k} e^{ i {\u k}\cdot {\u\psi}}, \qquad
\eta({\u\psi};\gamma) = \sum_{{\u k}\in \integer^2} \hat \eta_{\u k}(\gamma) e^{ i {\u k}\cdot {\u\psi}}\ ;
\eeq
inserting \equ{f-exp-g} into \eqref{cohom}, the equation \eqref{cohom} becomes
$$
\sum_{{\u k}\in \integer^2} \left( e^{2\pi i {\u \omega}\cdot{\u k}} - (1+\lambda) + \lambda e^{-2\pi i {\u \omega}\cdot
{\u k} } \right) \hat g_{\u k} e^{ i {\u\psi}\cdot {\u k}} - \mu = \sum_{{\u k}\in \integer^2} \hat \eta_{\u k}
(\gamma) e^{ i {\u \psi}\cdot {\u k}}\ .
$$
Its solution is given by \equ{sol-cohom-lemma}.
 {Note that $\hat{g}_{\u 0}$ is a free parameter; we choose it in such a way that $g$ satisfies the
normalization \equ{norm-g}, namely $\hat{g}_{\u 0}=0$.}
\end{proof}

\subsection{Expansions of the vector function $\u V$} \label{subsec:expansions}
To carry out the procedure described in Section
\ref{sec:iterative} to determine the Lindstedt series, one needs
to compute the formal expansion in \eqref{V1V2} for the function
{$\u V(\u\psi+\u P(\u\psi);\gamma)$, which can be found as follows
(see \cite{CellettiC95})}.

We consider the Fourier expansions
$$
\u V(\u\psi;\gamma) = \sum_{\u k\in \integer^2} {\hat{\u V}}_{\u k}(\gamma)\ e^{ i \u k \cdot \u\psi}\ ;
$$
therefore, we obtain
 {\begin{equation}\label{Vell}
    \u V(\u\psi + \u P(\u\psi);\gamma) = \sum_{\u k\in \integer^2} {\hat{\u V}}_{\u k}(\gamma)e^{ i \u k \cdot
    (\u\psi + \u P(\u\psi))}.
\end{equation}}
 {It is enough} to find the expansions of the exponential
functions as (compare with \cite{Brent})
\begin{equation}\label{expansion-exp-uv}
    e^{ i \u k \cdot (\u\psi + \u P(\u\psi))} = \sum_{j=0}^\infty b_{j,\u k}(\u\psi) \ep^j
\end{equation}
for suitable coefficients  {$b_{j,\u k}$} which are determined as
follows. Taking the derivative of both sides of
\eqref{expansion-exp-uv}, we obtain  {$$
     i \u k \cdot (\frac{d}{d\ep}\u P(\u\psi)) e^{ i \u k \cdot (\u\psi + \u P(\u\psi))} =
    \sum_{j=0}^\infty (j+1)b_{j+1,\u k}(\u\psi) \ep^j\ .
$$}
Using the formal expansion of $\u P$, together with \eqref{expansion-exp-uv}, we have
 {$$
 i \u k \cdot \left(\sum_{j=0}^\infty (j+1) \u P_{j+1}\ep^j \right) \sum_{j=0}^\infty b_{j,\u k} \ep^j  =
 \sum_{j=0}^\infty (j+1)b_{j+1,\u k} \ep^j\ ,
 $$
 namely
 $$
\sum_{j=0}^\infty  i \u k \cdot  \left(\sum_{h=0}^j (h+1) \u P_{h+1}\ b_{j-h,\u k}\right)  \ep^j  =
\sum_{j=0}^\infty (j+1)b_{j+1,\u k} \ep^j\ .
$$}
Thus, the following relations follow:
 {\beqano
    b_{0,\u k}(\u\psi) &=& e^{ i \u k \cdot \u\psi} \nonumber \\
    b_{j+1,\u k}(\u\psi) &=& \frac{ i}{j+1} \u k \cdot \sum_{h=0}^j (h+1) \u P_{h+1}(\u\psi)b_{j-h,\u k}(\u\psi)\ ,\qquad j\geq 0\ .
\eeqano}

\subsection{Explicit cohomology equations}

Putting together \equ{Vell} and \equ{expansion-exp-uv} one obtains
that the coefficient of order $\ep^j$ in  \equ{V1V2} is given by
$$
    \u V_j(\u \varphi; \gamma) = \sum_{\u k\in \integer^2} {\u{\hat {V}}}_{\u k}(\gamma) b_{j,\u k}(\u\psi).
$$
Thus, the explicit cohomology equations \eqref{eq-korder} at order $\ep^j$ are given by \begin{equation}\label{eq:explicit-cohom}
    \u P_j(\u\psi+2\pi\u\omega) - (I_2+\Lambda) \u P_j(\u\psi) +\Lambda \u P_j(\u\psi-2\pi\u\omega)
- \u\mu_j = \sum_{\u k\in \integer^2} {\u{\hat {V}}}_{\u k}(\gamma) b_{j-1,\u k}(\u\psi)
\end{equation}
whose solutions are obtained using Lemma \ref{lemma:sol-cohom-eq}.
Once  {$\u P_j$ and $\u  \mu_j$} are computed from
\eqref{eq:explicit-cohom}, one obtains the  {desired} formal
solutions  {$\u P$, $\u \mu$} of \equ{eqP}.

\subsection{Analyticity domains and breakdown of
attractors}\label{sec:domains}

In this Section, we provide some results for the  {2D and 4D}
standard maps introduced in Section~\ref{sec:SM}. Precisely, for a
fixed frequency vector, say $\u\omega$, we proceed along the
following steps:
\begin{itemize}
    \item[$A)$] we compute the Lindstedt series expansion up to a given order $N$ according to the procedure described in Section~\ref{sec:Lindstedt};
    \item[$B)$] we analyze the behavior of the coefficients of the Lindstedt series to provide an estimate
    of the radius of convergence;
    \item[$C)$] we determine the Pad\'e approximants (see Appendix~\ref{sec:Pade}) associated to the Lindstedt series to order $N$ and we draw the poles in the
    complex plane $(\varepsilon_r,\varepsilon_i)$.
\end{itemize}

Concerning step $B)$, we notice that the radius of convergence
$\rho$ of the Lindstedt series in the complex $\varepsilon$-plane
can be  {obtained as an approximation of}  {$$
\rho_u(\omega)=\inf_{\theta\in\torus} ({\rm limsup}_{n\to\infty}
|u_n(\psi)|^{1\over n})^{-1}\ ,\qquad
\rho_v(\omega)=\inf_{\theta\in\torus} ({\rm limsup}_{n\to\infty}
|v_n(\psi)|^{1\over n})^{-1}\ ,
$$
where $\u P_n = (u_n, v_n)$ are the coefficients of the formal expansion of $\u P$ in \eqref{param}. For the 2D standard map in \equ{DSM2} it is enough to consider a hull function $u$ satisfying
\beq{u2D}
x=\psi+u(\psi).
\eeq
} In practical computations, we will compute an approximation
{$\rho_u$, $\rho_v$} to the radius of convergence only for a
specific value of  {the parametric coordinate, say $\bar\psi=1$ in
the 2D case and $\u{\bar\psi}=(1,1)$ in the 4D case. In the
applications for steps B) and C), the estimate of the radius of
convergence is obtained by computing the limits for $n$ large of
the sequences $u_n({\bar\psi})^{-1/n}$, $v_n({\bar\psi})^{-1/n}$,
while the estimate of the breakdown threshold is obtained by
fitting the data to a circle and by computing the intersection of
the domain with the positive real axis. When the domain presents
oscillations, we indicate an inner and an outer circle. The
precision of the estimate of the threshold depends on the order of
the Pad\'e approximants and we will prefer to give the results as
an interval with a lower and upper limit of the threshold. We
remark that the results can in principle be sharpened by
increasing the order $N$ of the Lindstedt series expansions.}

\subsection{Radius of convergence and Pad\'e approximants for the 2D case}\label{sec:pade2dim}

 {In the following case studies,} the Lindstedt series in step $A)$ are determined
at  {different orders, up to $N=512$, and using up to 60} digits
of precision.

\begin{figure}[ht]
    \centering
    \includegraphics[trim = 8cm 10cm 8cm 10cm, width=3.4truecm]{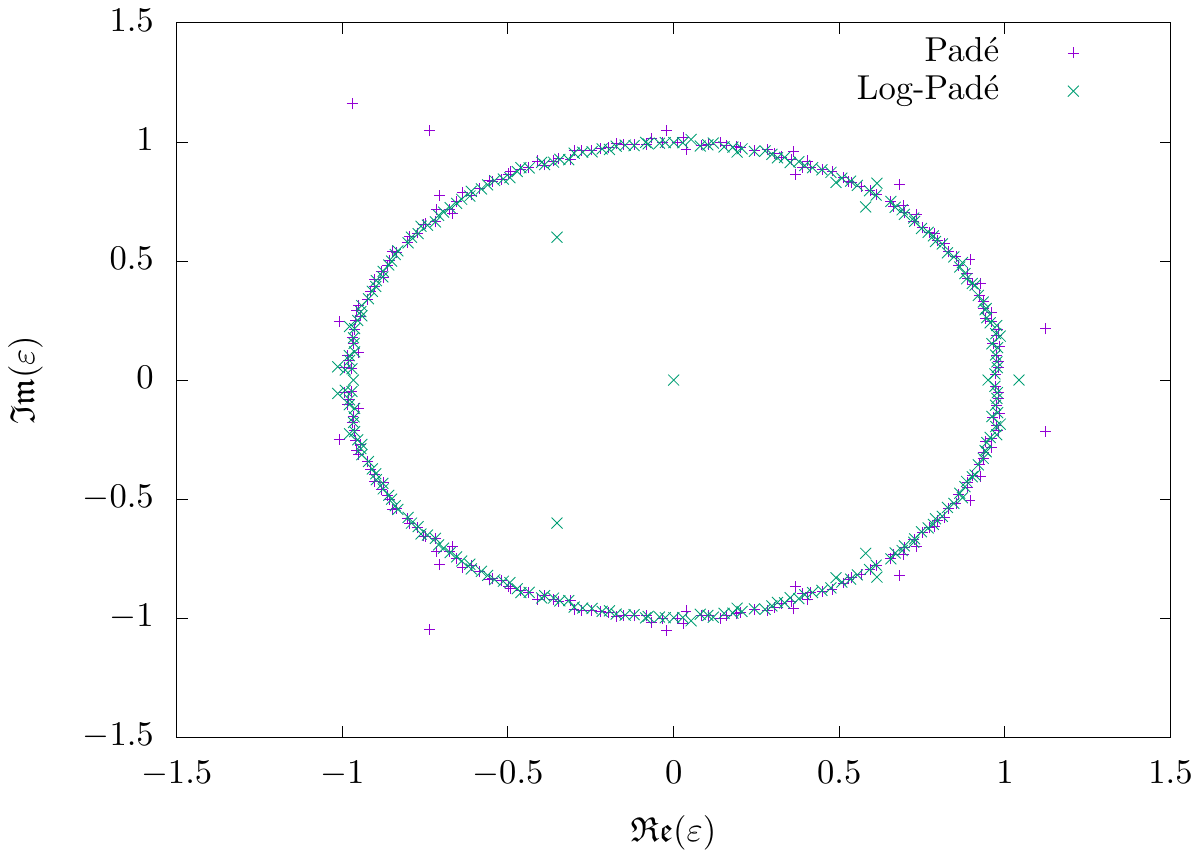}\hfil\hfil
    \includegraphics[trim = 8cm 10cm 8cm 10cm, width=3.4truecm]{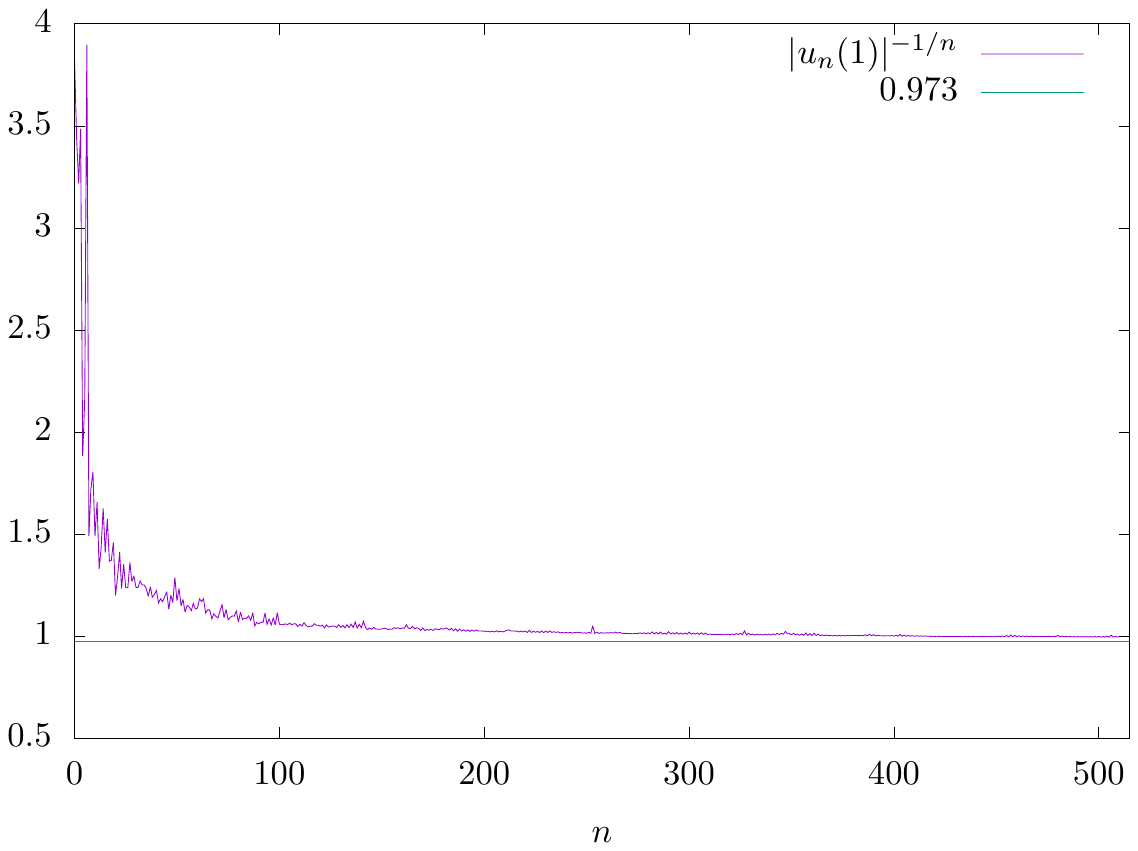}

    \caption{Map \equ{DSM2} with $V(x)=\sin x$, $\omega={{\sqrt{5}-1}\over {2}}$, $\lambda=0.8$.
    Left panel: poles of Pad\'e and log-Pad\'e approximants with $N=512$. Right panel: convergence of the Lindstedt series coefficients.}
    \label{fig:gold2dim}
\end{figure}

Figure~\ref{fig:gold2dim} refers to the map \equ{DSM2} with
$\omega={{\sqrt{5}-1}\over {2}}$, $\lambda=0.8$,  {$V(x)=\sin x$};
the left panel shows the  {zeros of the denominators of the}
Pad\'e approximants after computing the Lindstedt series up to the
order $N=512$ and the right panel shows the values
$|u_n(\bar\psi)|^{-{1\over n}}$ with $\bar\psi=1$ and
$n=1,...,512$. The breakdown threshold found in \cite{CallejaC10}
for this sample case was $\varepsilon=0.973$, which is consistent
with the radius of convergence in the right panel of
Figure~\ref{fig:gold2dim} and with the threshold determined in the
left panel of Figure~\ref{fig:gold2dim} as the intersection of the
analyticity domain with the positive abscissa (compare with
\cite{BCCF}, \cite{Lla-Tom-95}).

It is known that the Pad\'e approximants are better suited to
approximate functions with poles, rather than functions with
branch points (\cite{Bak-96}). The theoretical and numerical
evidence presented in \cite{Lla-Tom-94}, \cite{Lla-Tom-95}
indicates that the boundaries of the domains of analyticity of
invariant circles could be described as an accumulation of branch
points. As a consequence, a logarithmic Pad\'e approximation
{(see Figure~\ref{fig:gold2dim}, left panel)} is better suited to
compute singularities that are branch points; we follow
\cite{Lla-Tom-95} for the computation of the logarithmic Pad\'e
approximation. More precisely, if a function $f(z)$ has a  branch
point singularity at $z=1/\alpha$, then
$$
f(z) = C(1-\alpha z)^c +g(z)\ ,
$$
where $g(z)$ is  {of higher order} at $z=1/\alpha$,  {$C$ is a
constant and $c< 0$.} For $z$ close to $1/\alpha$,  one has
$$ F(z) := \frac{d}{dz} \ln(
f(z) ) = \frac{f'(z)}{f(z)}\approx\frac{\chi}{z-(1/\alpha )}\ ,
$$
for some $\chi\in\real$. The order $\chi$ of the branch point could be estimated by using
Pad\'e approximants on different functions related to $F(z)$. For example,
close to $z=1/\alpha$, one has $$\left(z - \frac{1}{\alpha}\right)\frac{f'(z)}{f(z)} \approx \chi. $$

One expects that a Pad\'e approximant for $F(z)$
exhibits a pole at $z=1/\alpha$ with residue $\chi$. Finally, to get an
$[M,N]$ Pad\'e approximant of $F(z)$ one needs to find polynomials
$Q_1^{(M)}(z)$ and $Q_2^{(N)}(z)$ of degrees $M$ and $N$,
respectively, satisfying
$$
\frac{f'(z)}{f(z)} =\frac{Q_1^{(M)}(z)}{Q_2^{(N)}(z)} + O(z^{M+N+1})
$$
and $Q_2^{(N)}(0)=1.$

\begin{figure}[ht]
    \centering
    \includegraphics[trim = 8cm 10cm 8cm 10cm, width=3.4truecm]{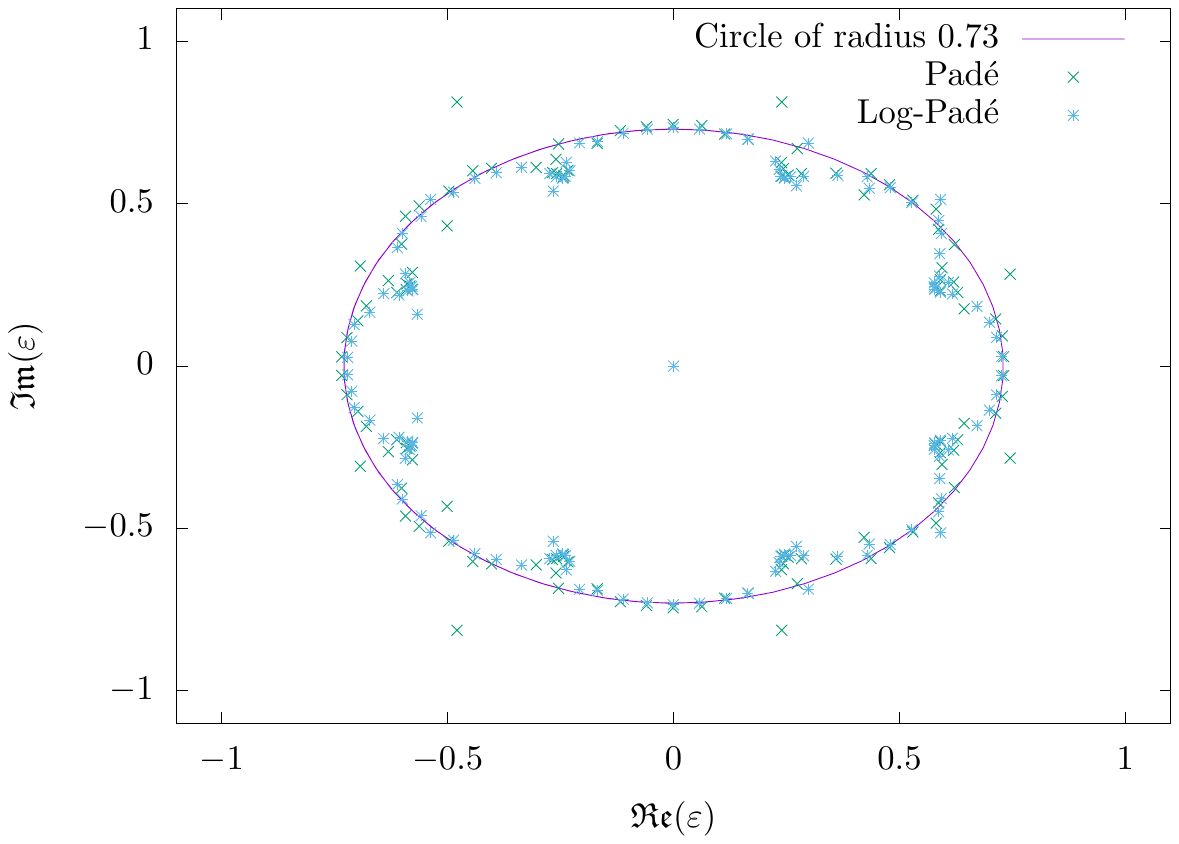}\hfil\hfil
    \includegraphics[trim = 8cm 10cm 8cm 10cm, width=3.4truecm]{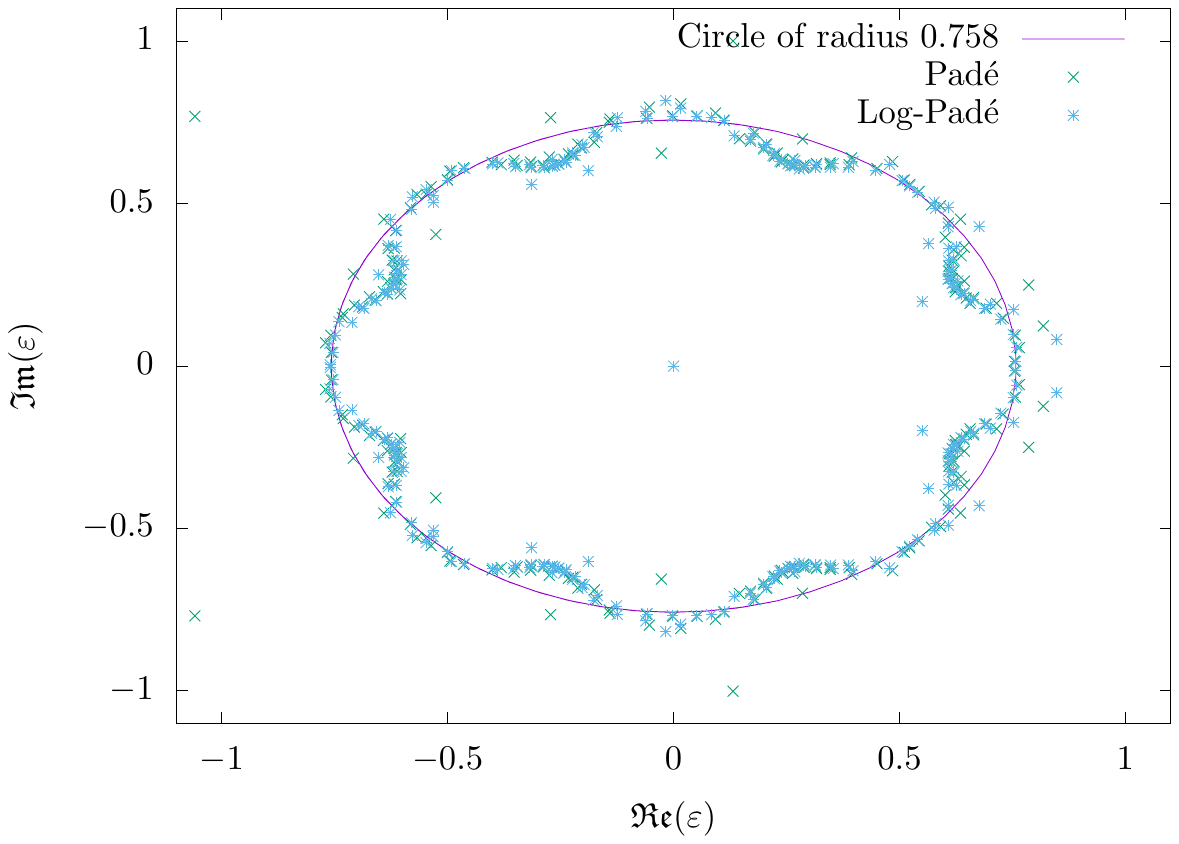}\\
    \includegraphics[trim = 8cm 10cm 8cm 9cm, width=3.4truecm]{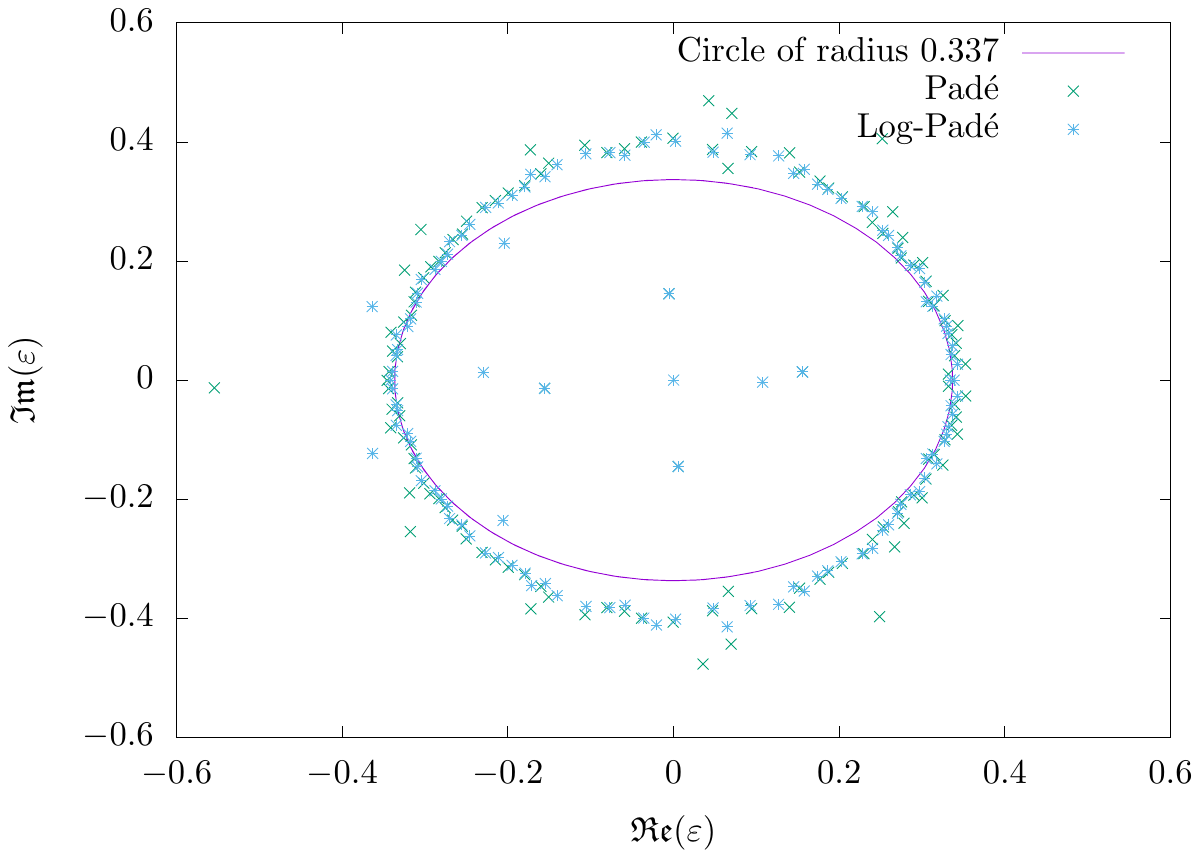}\hfil\hfil
    \includegraphics[trim = 8cm 10cm 8cm 9cm, width=3.4truecm]{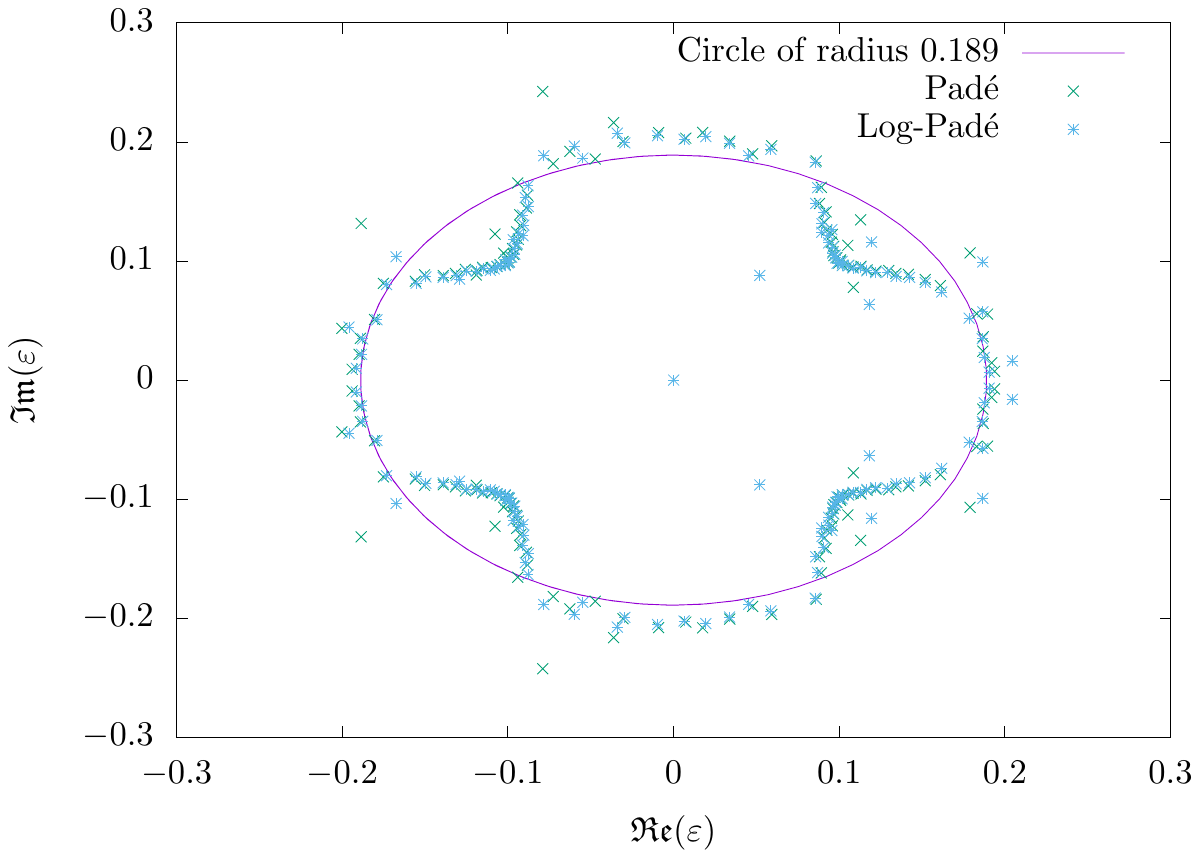}

    \caption{Poles of Pad\'e and log-Pad\'e approximants using Lindstedt series of order at least $N=256$.
Upper plots: map \equ{DSM2} with $V(x)=\sin(x)$, $\omega=s^{-1}$.
Left: $\lambda=0.9$, and circle of radius 0.73. Right:
$\lambda=0.8$ and circle of radius 0.758. Lower plots: map
\equ{DSM2} with $V_{2h}(x)=\sin x + {1\over 3}\sin(3x)$,
$\lambda=0.8$. Left: $\omega={{\sqrt{5}-1}\over {2}}$ and circle
of radius 0.337. Right: $\omega=s^{-1}$ and circle of radius
{0.189}.} \label{fig:1os2dim}
\end{figure}

Figure~\ref{fig:1os2dim}, upper plots, refers  {again to the 2D
map in \equ{DSM2}, but with} the frequency $\omega=s^{-1}$; the
left panel shows  {the zeros of the denominators of} the Pad\'e
and log-Pad\'e approximants for the  one-harmonic case $V(x)=\sin
x$ and for two values of $\lambda$;  {the circles shown in
Figure~\ref{fig:1os2dim} are used to approximate the breakdown
threshold} and they are estimated as $\varepsilon_c {\approx}0.73$
for $\lambda=0.9$ and $\varepsilon_c {\approx}0.758$ for
$\lambda=0.8$, well below the threshold of the golden ratio of
Figure~\ref{fig:gold2dim}. Besides, for $\omega=s^{-1}$ the
results for the domain take a sort of flower shape. The results
for the two-harmonic case with $V_{2h}(x)=\sin x+{1\over
3}\sin(3x)$ are shown in the lower panels of
Figure~\ref{fig:1os2dim} for $\lambda=0.8$ with
$\omega={{\sqrt{5}-1}\over {2}}$ (left) and $\omega=s^{-1}$
(right). In the latter case the estimate of the threshold is more
difficult, due to the irregular shape of the domain of
analyticity.

To analyze the dependence of the domains of analyticity on the
choice of the dissipation parameter $\lambda$,
Figure~\ref{fig:2arm} shows the results of the  {poles of the}
Pad\'e approximants for the two harmonics potential $V_{2h}$ and
for different values of $\lambda$, taking
$\omega={{\sqrt{5}-1}\over {2}}$ (left panel) and $\omega=s^{-1}$
(right panel). We fixed $N=256$ and used 40 digits of precision;
 {notice that the lines of poles in Figure~\ref{fig:2arm} are
known to be  an indication of branch singularities (see
\cite{Lla-Tom-94}, \cite{Lla-Tom-95} for further details)}.

\begin{figure}[ht]
    \centering
    \includegraphics[trim = 8cm 10cm 8cm 10cm, width=3.4truecm]{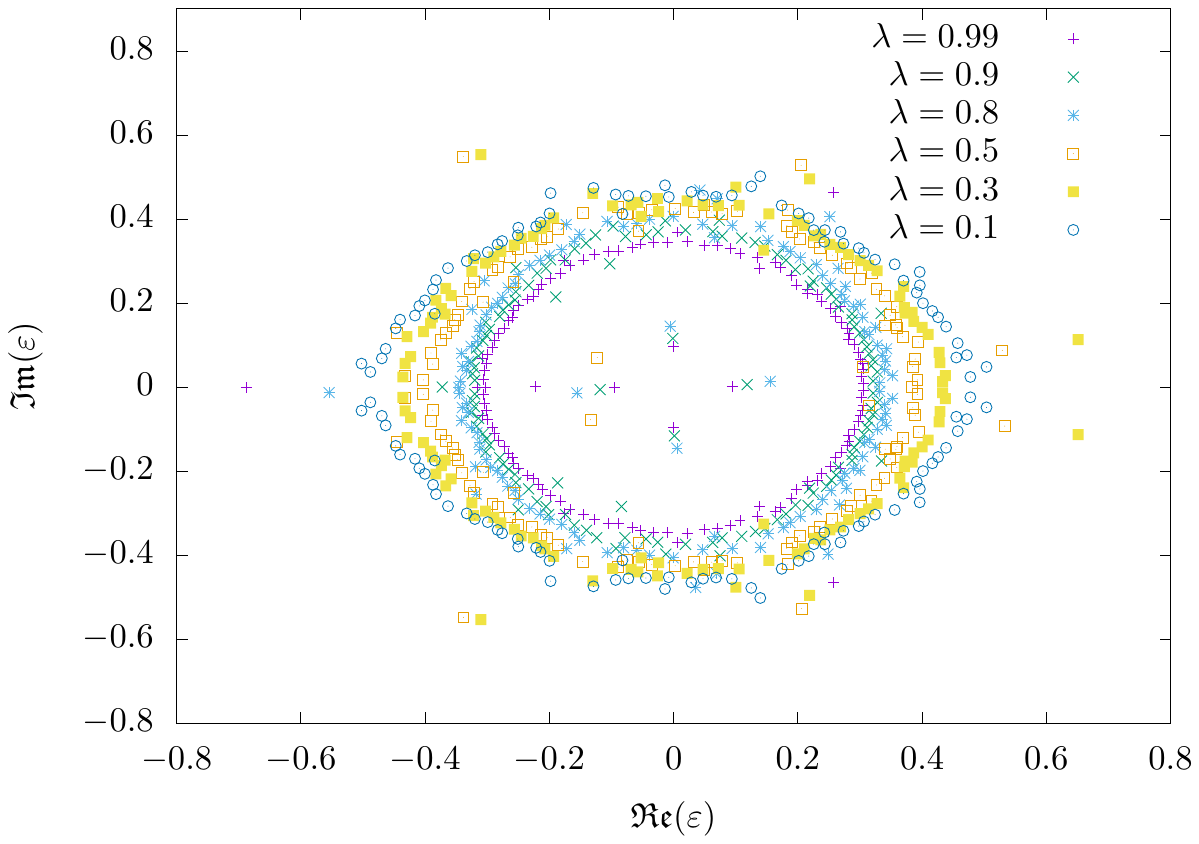}\hfil\hfil
    \includegraphics[trim = 8cm 10cm 8cm 10cm, width=3.4truecm]{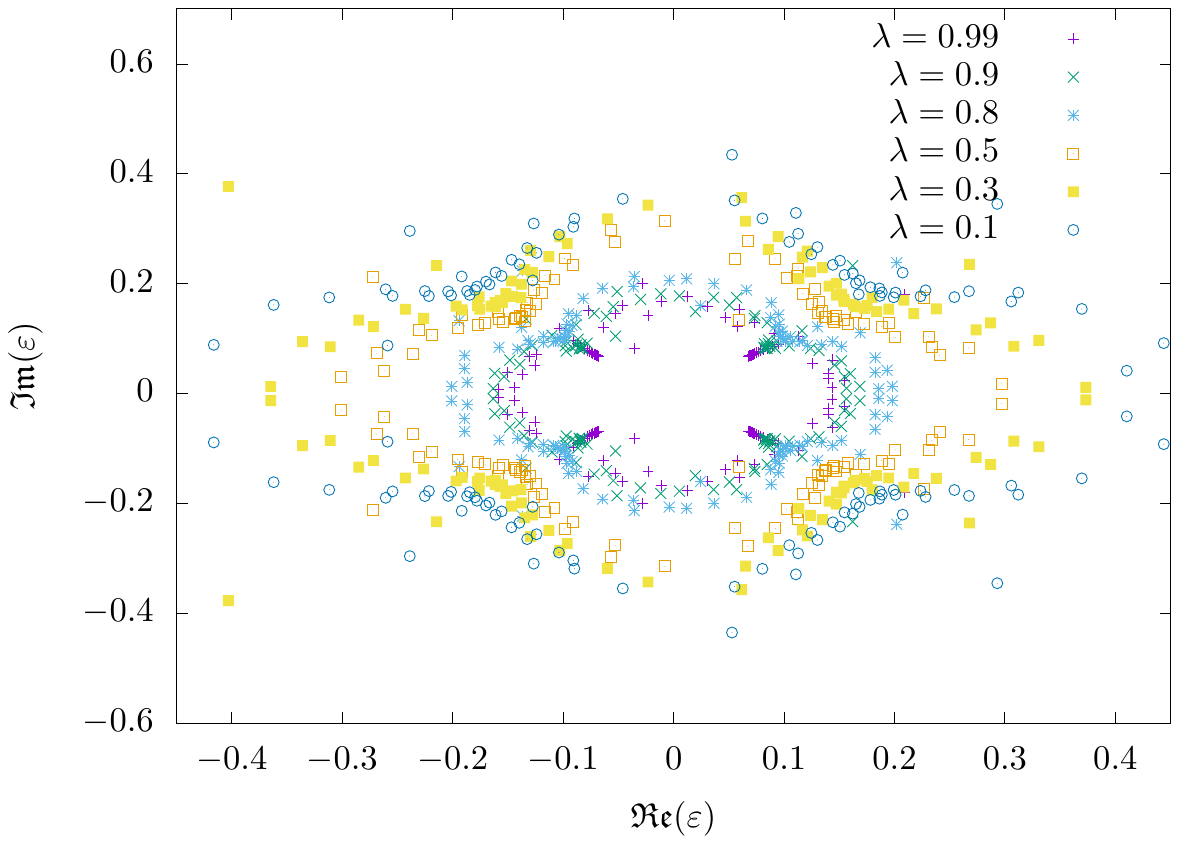}

    \caption{Map \equ{DSM2} with $V_{2h}(x)=\sin x+{1\over 3}\sin(3x)$ for different values of $\lambda$
with $N=256$ and 40 digits of precision. The plots display the poles of the Pad\'e approximants for $N=256$.
Left panel: $\omega={{\sqrt{5}-1}\over {2}}$.
Right panel: $\omega=s^{-1}$.}
\label{fig:2arm}
\end{figure}

\subsection{Pad\'e approximants for the 4D case}\label{sec:pade4dim}

For the 4D maps  {introduced in \equ{SMS} with the potential in
\equ{W}}, the computations are definitely more demanding; for this
reason, the Lindstedt series in step $A)$ are determined only up
to order $N=128$, and we reach in some cases $N=256$.
 {Figure~\ref{fig:4d-sympl-case} refers to the case with
$\u\omega_s$, $\gamma = 0.01$; they show the poles of the Pad\'e
and log-Pad\'e approximants with $N=128$, obtained using 100
digits of precision, the lower right panel shows the radius of
convergence of the Lindstedt coefficients. Given that the poles of
the Pad\'e approximants often lie on an irregular curve, we
provide an inner and outer circle to confine the regions where
most of the poles are located; due to the asymptotic character of
the Lindstedt series, we expect that the radius of convergence of
the Lindstedt series occurs within the two level lines
corresponding to the radii of the inner and outer circles as shown
in Figure~\ref{fig:4d-sympl-case}, lower right panel.}

\begin{figure}[ht]
    \centering
    \includegraphics[trim = 8cm 10cm 8cm 10cm, width=3.4truecm]{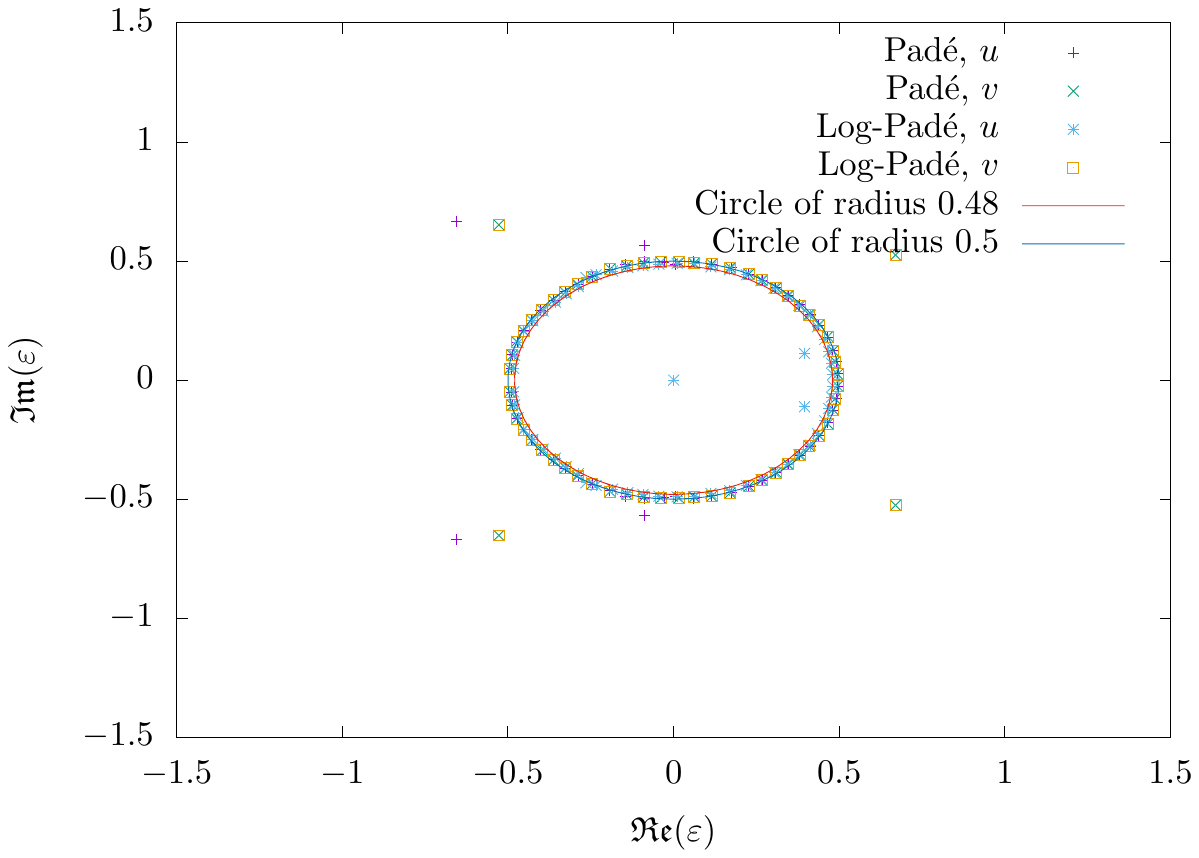}\hfil\hfil
    \includegraphics[trim = 8cm 10cm 8cm 10cm, width=3.4truecm]{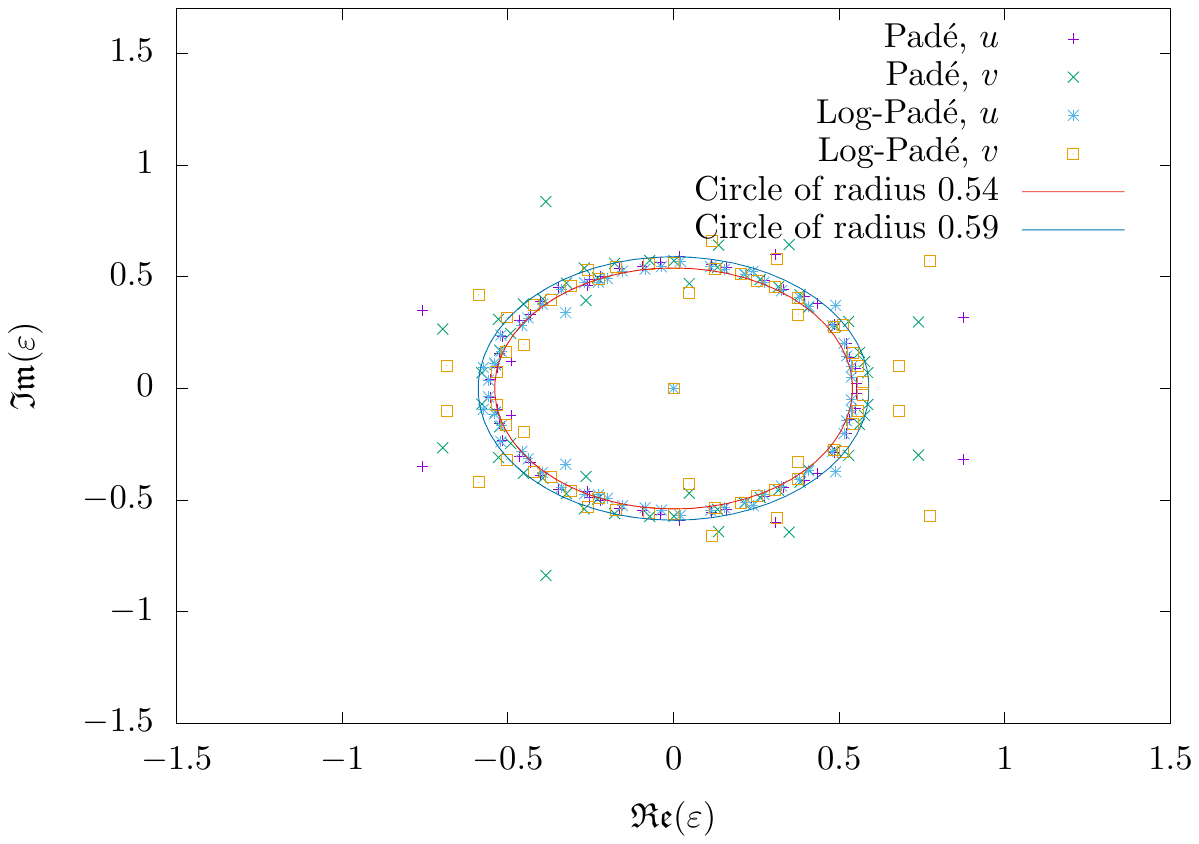}\\
    \includegraphics[trim = 8cm 10cm 8cm 10cm, width=3.4truecm]{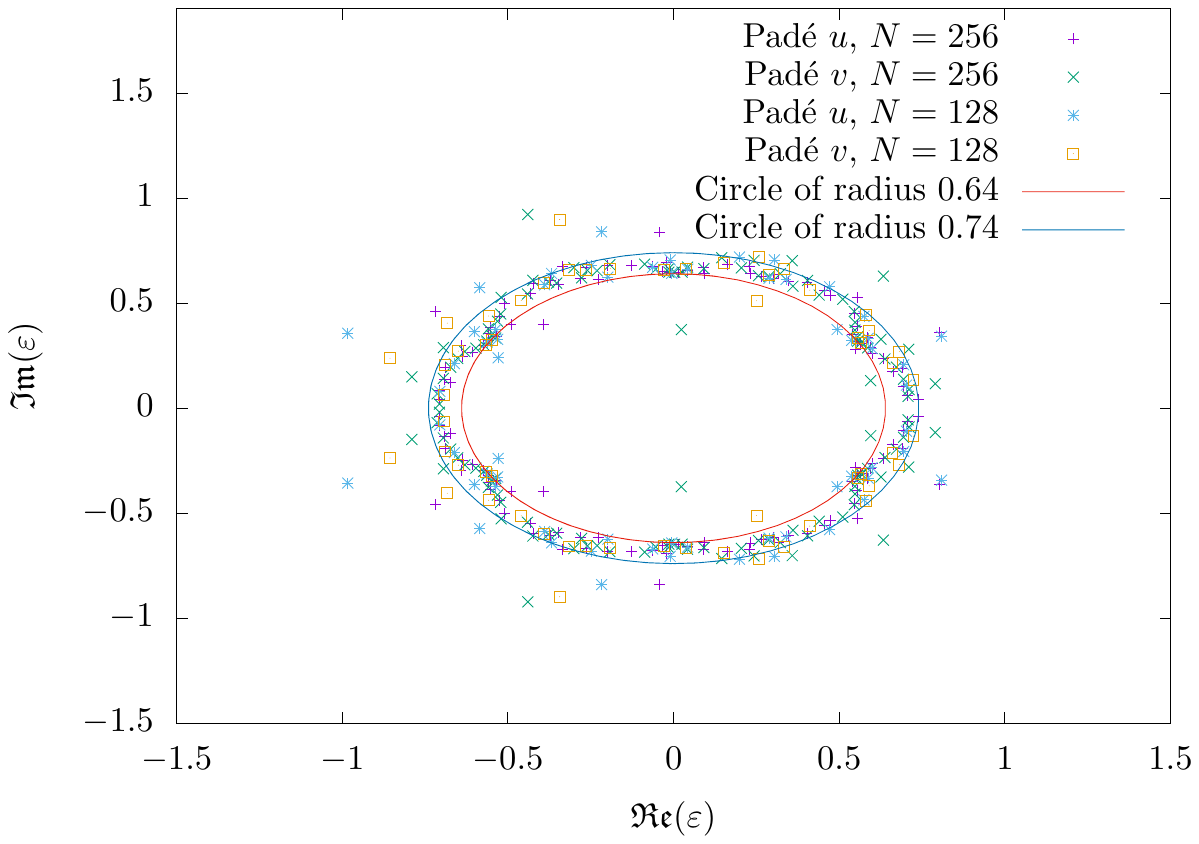}\hfil\hfil
    \includegraphics[trim = 8cm 10cm 8cm 10cm, width=3.4truecm]{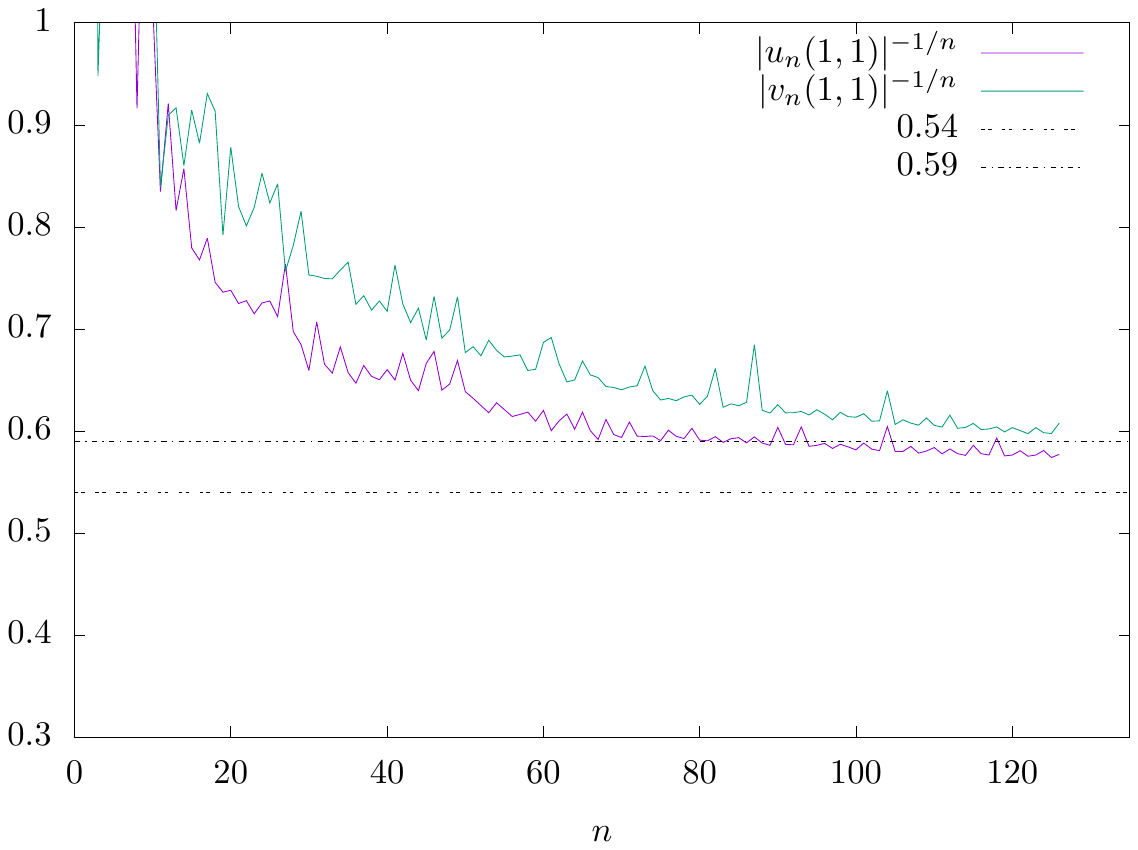}
    \caption{ {Map \equ{SMS} with $\lambda_1=\lambda_2=1$ (upper left), $\lambda_1=1$, $\lambda_2=0.8$ (upper right),
    $\lambda_1=0.8$, $\lambda_2=0.7$ (lower left) for $\u\omega_s$,
$\gamma = 0.01$. The plots display the poles of the Pad\'e and log-Pad\'e approximants for $N=128$
 {or (lower left) the poles of the Pad\'e approximants for $N=128$ and $N=256$}.
The lower right panel  {shows the plots for the estimate of the}
radius of convergence of the Lindstedt series using coefficients
up to order $N = 128$ and $\u\psi = (1,1)$ for $\lambda_1=1$,
$\lambda_2=0.8$.}}\label{fig:4d-sympl-case}
\end{figure}


We studied also the  {other frequency vectors introduced in
\equ{vectors}}, providing an estimate of the breakdown threshold
by looking at the  {intersection of the poles of the Pad\'e
approximants with the real axis}. The results of such thresholds
are summarized in Table~\ref{tab:omegas};  {also  in this case, we
provide results in the form of an interval giving the radii of the
annular regions containing most of the poles.
Table~\ref{tab:omegas}} shows that the  {mean values of the
intervals often} tend to increase, as the 4D map switches from
symplectic to mixed to dissipative.

\vskip.1in

 {
\begin{table}
$$
\begin{array}{|c|c|c|c|c|c|c|}
\hline
 & {\u\omega}_s & {\u\omega}_u & {\u\omega}_\tau & {\u\omega}_g & {\u\omega}_a& {\u\omega}_c\\
\hline
{\rm Symplectic} & 0.48-0.50 & 0.48-0.51 & 0.48-0.51 & 0.52-0.56 & 0.45-0.48 & 0.37-0.40\\
{\rm Mixed} &      0.54-0.59 & 0.46-0.50 & 0.47-0.53 & 0.65-0.70 &0.44-0.48&0.49-0.53\\
{\rm Dissipative} &0.64-0.74 & 0.65-0.69 & 0.59-0.64 & 0.67-0.76& 0.63-0.69 & 0.64-0.69 \\
\hline
\end{array}
$$
\caption{Radii of the annular regions that encompasses most of the poles of the Pad\'e and log-Pad\'e
approximants for different frequency vectors.
The approximants were computed using Lindstedt series of order at least  $N=128$.
 {Map \equ{SMS}-\equ{W} with $\gamma=0.01$.
Dissipative: $\lambda_1 = 0.8$, $\lambda_2 = 0.7$. Mixed: $\lambda_1=1$, $\lambda_2 = 0.8$.
Symplectic: $\lambda_1= 1$, $\lambda_2 = 1$.}}\label{tab:omegas}
\end{table}
}

Furthermore, we provide in Table~\ref{tab:gammas} some results
concerning the dependence of the thresholds on the coupling
parameter. Precisely, we consider the frequency $\u\omega_s$ and
we vary the coupling parameter $\gamma$ on a set of values  {from
$10^{-4}$ to 0.5. We provide the breakdown threshold  {again as an
interval} computed by looking at the poles of the Pad\'e and
log-Pad\'e approximants, further validated by the computation of
the radius of convergence of the Lindstedt coefficients. As it is
expected, the values decrease as the coupling parameter
increases.}

\begin{table}[h]
 {$$
\begin{array}{|c|c|c|c|}
\hline
 \gamma & {\rm Symplectic} & {\rm Mixed} & {\rm Dissipative} \\
\hline
0.0001 & 0.57-0.64 & 0.57-0.71 &  0.64-0.78  \\
0.0005 & 0.56-0.62 & 0.56-0.68 & 0.64-0.78 \\
0.001 & 0.56-0.58 & 0.56-0.64 & 0.64-0.78  \\
0.005 & 0.50-0.52 & 0.55-0.59 & 0.64-0.74    \\
0.01 & 0.48-0.50 & 0.54-0.59  & 0.64-0.73   \\
0.05 & 0.43-0.45 & 0.45-0.52 & 0.61-0.66   \\
0.1 & 0.40-0.43 & 0.40-0.50 & 0.57-0.63 \\
0.25 & 0.32-0.38 & 0.33-0.40 & 0.46-0.54  \\
0.5 & 0.21-0.29 & 0.26-0.31 & 0.35-0.44 \\
\hline
\end{array}
$$}
\caption{ {Radii of the annular regions that encompasses most of
the poles of the Pad\'e and log-Pad\'e approximants for the
 {map \equ{SMS}-\equ{W},} the frequency $\u\omega_s$ and for
different values of the coupling parameter $\gamma$. The
approximants were computed using Lindstedt series of  order at
least  $N=128$. Dissipative: $\lambda_1 = 0.8$, $\lambda_2 = 0.7$.
Mixed: $\lambda_1=1$, $\lambda_2 = 0.8$. Symplectic: $\lambda_1=
1$, $\lambda_2 = 1$.}} \label{tab:gammas}
\end{table}

We conclude by mentioning some results for potentials with two
harmonics and a coupling, say $V_{1hc}(x, z) = \sin(x) +
\frac{1}{3}\sin(3x) +\gamma\sin(x - z)$, $V_{2hc}(x, z) = \sin(z)
+ \frac{1}{3}\sin(3z) - \gamma\sin(x - z)$. The plots of
Figure~\ref{fig:omega-tau} show the results for the map \equ{SMS}
in the symplectic case with frequency $\u\omega_s$ and
$\gamma=0.01$. The poles of the Pad\'e and log-Pad\'e approximants
are computed up to the order $N=85$ with a precision of at least
50 digits. The  {estimate of the breakdown threshold is definitely
complicated by the very irregular behaviour of the poles that
appears in the two-harmonic potential.}

\begin{figure}[ht]
\centering
\includegraphics[trim = 8cm 9cm 8cm 10cm, width=3.4truecm]{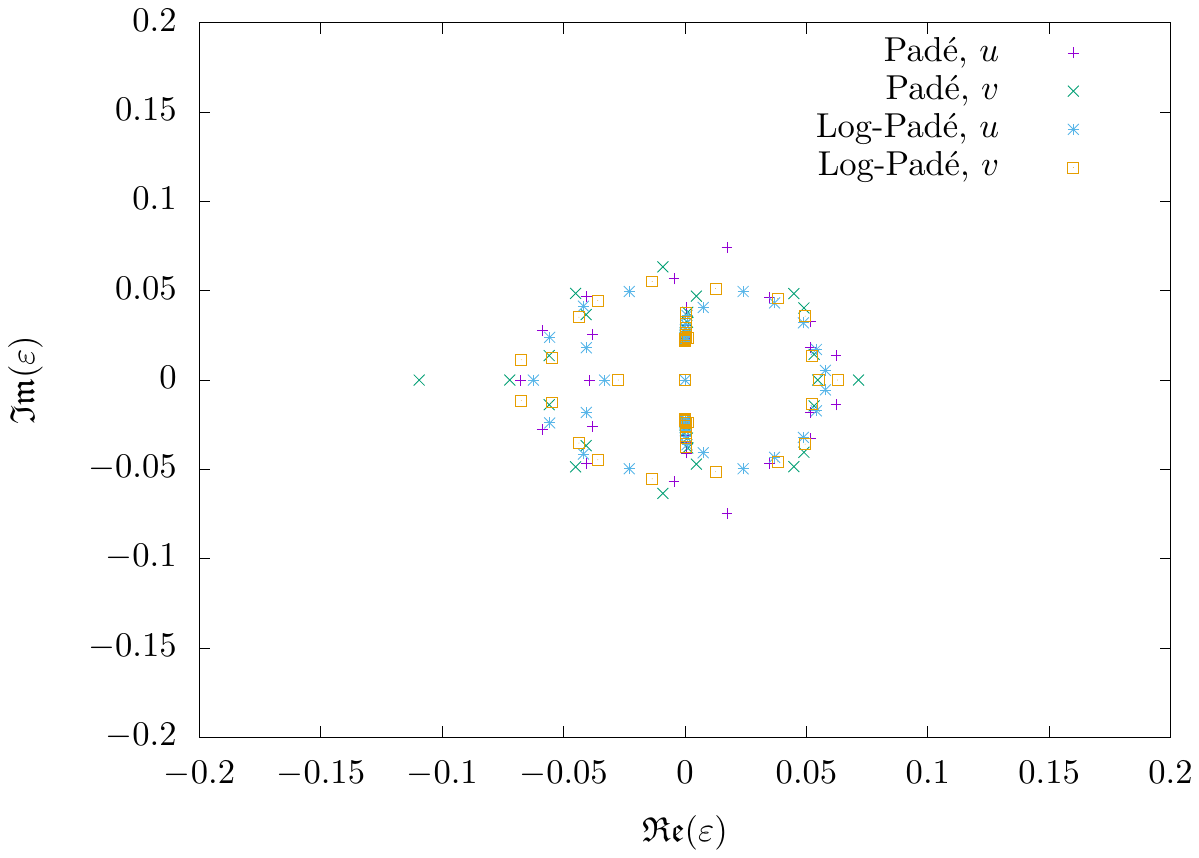}\hfil\hfil
\includegraphics[trim = 8cm 9cm 8cm 10cm, width=3.4truecm]{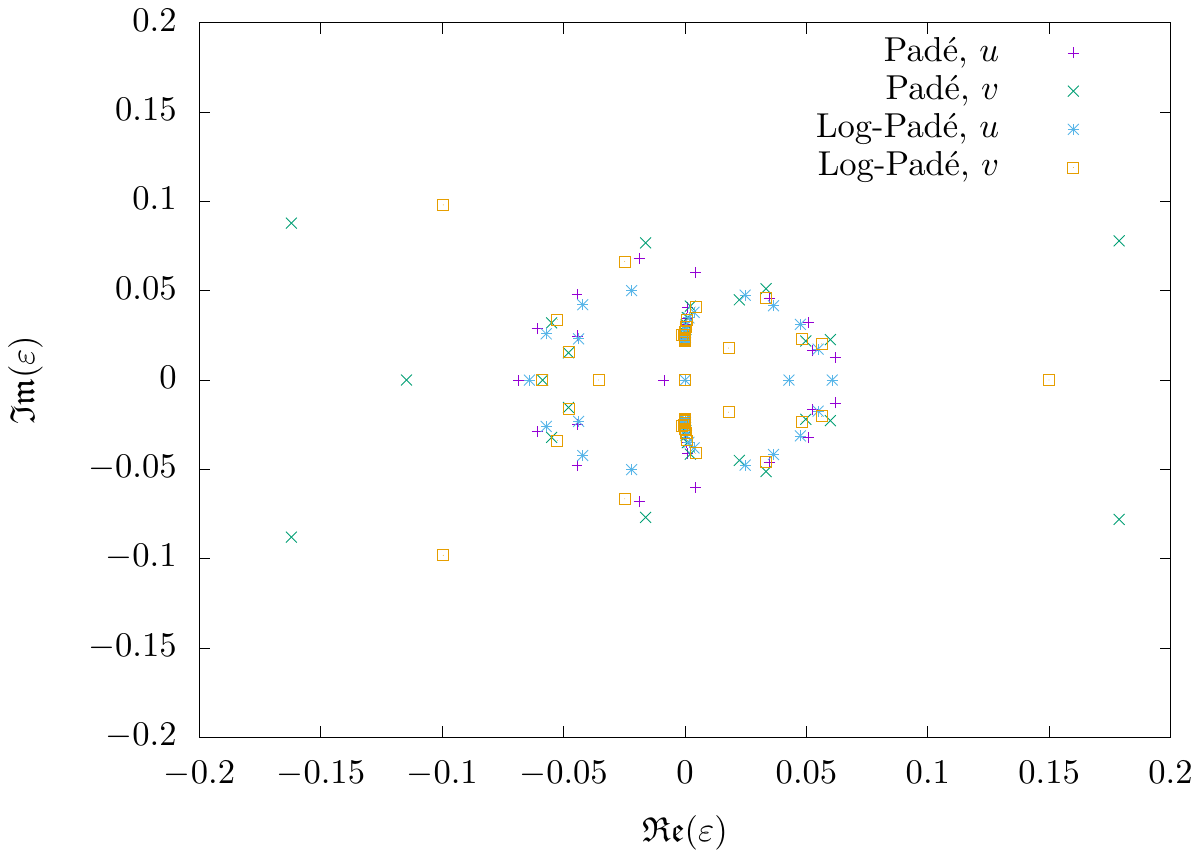} \\
\includegraphics[trim = 8cm 10cm 8cm 10cm, width=3.4truecm]{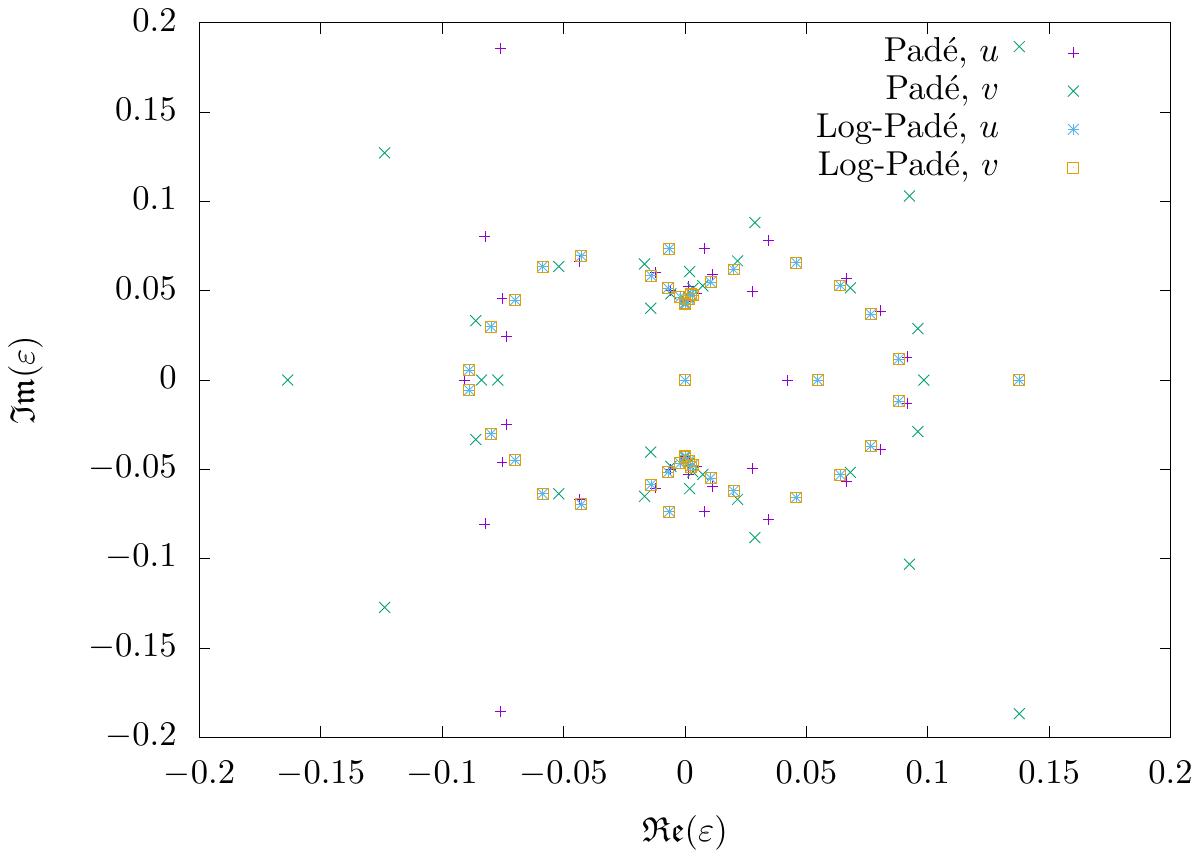}\hfil\hfil
\includegraphics[trim = 8cm 10cm 8cm 10cm, width=3.4truecm]{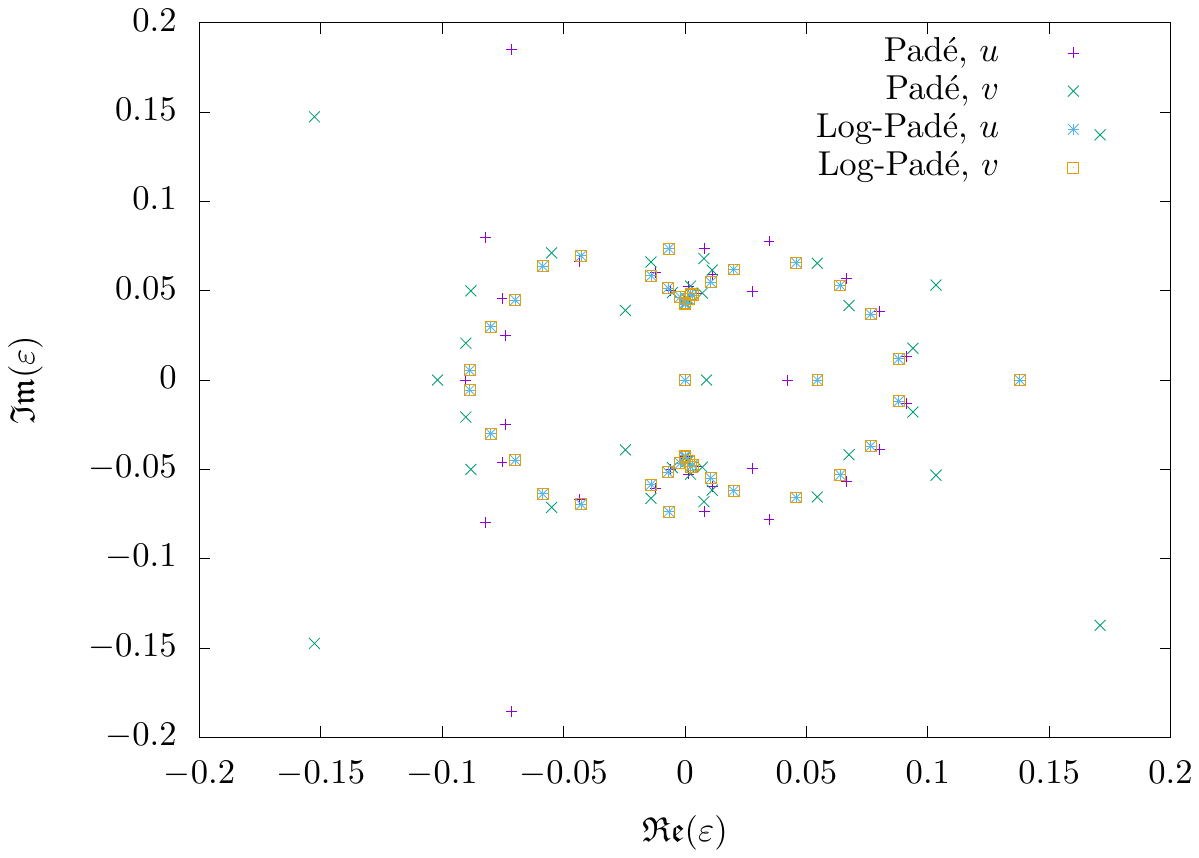}
\caption{Map \equ{SMS} with the 2-harmonic potential $V_{1hc}(x, z) = \sin(x)
+ \frac{1}{3}\sin(3x) +\gamma\sin(x - z)$,
$V_{2hc}(x, z) = \sin(z) + \frac{1}{3}\sin(3z)
- \gamma\sin(x - z)$, $\u\omega_s$,
$\gamma = 0.01$. Poles of Pad\'e and log-Pad\'e approximants using Lindstedt series of order $N=85$.
Upper left panel: $\lambda_1 = 1$, $\lambda_2=1$.
Upper right panel: $\lambda_1 = 1$, $\lambda_2=0.8$.
Lower left panel: $\lambda_1 = 0.8$, $\lambda_2=0.8$.
Lower right panel: $\lambda_1 = 0.8$, $\lambda_2=0.7$.}
\label{fig:omega-tau}
\end{figure}

\section{Newton method and Sobolev criterion}\label{sec:newton}

This Section is devoted to the presentation of Newton method for
the 2D  {(Section~\ref{sec:newton2d})} and 4D cases
 {(Section~\ref{sec:newton4d})}; we also describe the Sobolev
criterion in 4D  {(Section~\ref{sec:Sobolev}, see
\cite{CallejaC10} for the Sobolev criterion in 2D)} to find the
breakdown threshold and we provide some results for the 2D
 {(Section~\ref{sec:newres2d})} and 4D maps
 {(Section~\ref{sec:newres4d})}.

\subsection{Newton method for 2D maps}\label{sec:newton2d}
For the dissipative 2D standard map \eqref{DSM2}, we use a Newton
method that has been implemented very successfully in
\cite{CallejaC10}. It consists in finding a hull function  {$\u
P:\torus \longrightarrow \real$} and a drift parameter $\mu\in
\real$ satisfying equation  {\equ{eqP}}. The details about the
construction and implementation of the Newton method  {in the 2D
case} can be found in \cite{CallejaC10}.

\subsection{Newton method for 4D maps}\label{sec:newton4d}
In the 4D case, we present a procedure borrowed from \cite{CallejaCdlL2013}
to find, through a Newton method, the parameterization of an
invariant torus. In the following lines we give a brief
description of the construction of the corrections at each step of
the Newton method.

Consider the map  {$\u f_{\u\mu}: \torus^2\times \real^2
\longrightarrow \torus^2 \times \real^2$ given by \equ{SMS}.} In
this Section, we limit to consider a family of conformally
symplectic maps ${\u f}_{\u \mu}$ with conformal factor $\lambda$
equal for both components $y$ and $w$. We remark that the
extension to the case of different conformal factors, say
$\lambda_1$, $\lambda_2$, cannot be directly derived from
\cite{CallejaCdlL2013},  {but one could rather use one of the many
algorithms implemented and run in \cite{CH2}}.

We look for a parameterization ${\u K}:\torus^2\longrightarrow
\torus^2\times \real^2$ and a constant vector ${\u \mu}\in\real^2$
satisfying the invariance equation  {\equ{inv}.} Note that if such
${\u K}$ and ${\u \mu}$ satisfying  {\equ{inv}} exist, then there
exists an invariant torus, parameterized by ${\u K}$, whose
dynamics is conjugated to a rotation by ${\u \omega}$.  {Notice
also that $\u K$ needs to satisfy a normalization condition, due
to the fact that the solution of \equ{inv} is not unique
(\cite{CallejaCdlL2013}).}

The Newton method we use takes advantage of the so-called
``automatic reducibility": in a neighborhood of an invariant torus
there exists an explicit change of coordinates that makes the
linearization of the invariance equation \eqref{inv}
into an equation with constant coefficients (see equation
\equ{system} below). That is, defining  $M({\u \theta})$ as the
4$\times$4 matrix \begin{equation}\label{M}
    M({\u \theta}) = [D{\u K}({\u \theta}) | J^{-1}D{\u K}({\u \theta})N({\u \theta})]\ ,
\end{equation}
where $N({\u \theta}) = [D{\u K}({\u \theta})^\top D{\u K}({\u \theta})]^{-1}$, then,
if ${\u K}$ and ${\u \mu}$ satisfy \eqref{inv}, one has that
\begin{equation}\label{aut-red}
    D{\u f}_{{\u \mu}}\circ {\u K}({\u \theta}) M({\u \theta}) =
    M({\u \theta}+2\pi {\u \omega})\left(\begin{array}{cc} Id & S({\u \theta})  \\ 0 & \lambda Id
    \end{array} \right)\ ,
\end{equation}
where $S({\u \theta})$ is an explicit function depending on $D{\u
K}({\u \theta})$ and $D{\u f}_{{\u \mu}}\circ {\u K}({\u \theta})$
{(see \cite{CallejaCdlL2013} for the explicit formulation of the
function $S$)}.

Next, we briefly describe the quasi-Newton method used to find
solutions of \eqref{inv}; a detailed exposition of such
method can be found in \cite{CallejaCdlL2013}. We start with an approximate
solution of \eqref{inv}, say
$$
    {\u f}_{{\u \mu}}\circ {\u K} - {\u K}\circ T_{\u \omega} = {\u E}\ ,
$$
 {where $T_{\u \omega}({\u \psi}) = {\u \psi} +2\pi {\u \omega}$ and} the error ${\u E}$ in the approximation is supposed to be
small. The Newton method introduced in \cite{CallejaCdlL2013} consists in
finding corrections $\u \Delta$ and $\u \sigma$ to $\u K$ and $\u
\mu$, respectively, such that the approximate solution of
\eqref{inv} associated to ${\u K}+{\u \Delta}$, ${\u
\mu} + {\u \sigma}$ quadratically reduces the error. Taking into
account that
$$
    {\u f}_{{\u \mu} + {\u \sigma}} \circ ({\u K}+{\u \Delta}) = {\u f}_{{\u \mu}}\circ {\u K} +
    [D{\u f}_{{\u \mu}}\circ {\u K}]{\u \Delta} + [D_{\u \mu} {\u f}_{{\u \mu}}\circ {\u K}]
    {\u \sigma} + O(\|{\u \Delta}\|^2) + O(\|{\u \sigma}\|^2)\ ,
$$
the Newton method consists in finding $\u \Delta$ and $\u \sigma$
satisfying \begin{equation}\label{newton-eq}
    [D{\u f}_{{\u \mu}}\circ {\u K}]{\u \Delta} -{\u \Delta}\circ T_{\u \omega} + [D_{\u \mu}
    {\u f}_{{\u \mu}}\circ {\u K}]{\u \sigma} = - {\u E}\ .
\end{equation}
Instead of solving \eqref{newton-eq}, the main idea in \cite{CallejaCdlL2013} is to use the automatic reducibility
to find an approximate solution of \eqref{eqP}, that still leads to a quadratically convergent procedure. Using the matrix $M$
defined in \eqref{M}, a change of variables is introduced by setting
$$
    {\u \Delta} = M{\u W}\ .
$$
In the new unknowns ${\u W}$ and ${\u \sigma}$, equation \eqref{newton-eq} transforms into
$$
    [D{\u f}_{{\u \mu}}\circ {\u K}]M{\u W} -(M\circ T_{\u \omega})({\u W}\circ T_{\u \omega}) + [D_{\u \mu} {\u f}_{{\u \mu}}\circ
    {\u K}]{\u \sigma} = -{\u E}\ .
$$
Then using \eqref{aut-red}, and ignoring an error term in \eqref{aut-red} coming from the fact that ${\u K}$, ${\u \mu}$
is an approximate solution to \eqref{inv}, one obtains
\begin{equation}\label{quasi-newton}
    M\circ T_{\u \omega} \left[ \left(\begin{array}{cc} Id & S(\theta) \\ 0 & \lambda Id \end{array}\right){\u W} -
    {\u W}\circ T_{\u \omega} \right] + \left[ D_{\u \mu} {\u f}_{{\u \mu}}\circ {\u K}\right]{\u \sigma} = -{\u E}\ .
\end{equation}
As it is noted in \cite{CallejaCdlL2013}, equation \eqref{quasi-newton} reduces to difference equations with constant coefficients,
so that it can be solved efficiently by using Fourier methods. Using the notation
${\u W}=(W_1,W_2)^\top$, $\tilde{{\u E}} = (M^{-1}\circ T_{\u \omega}) {\u E}:= (\tilde{{\u E}}_1, \tilde{{\u E}}_2)^\top $, and
$\tilde{A}
= (M^{-1}\circ T_{\u \omega})D_{\u \mu} {\u f}_{{\u \mu}}\circ {\u K} := [\tilde{A}_1|\tilde{A}_2] $,
equation \eqref{quasi-newton} can be expressed in components as
\begin{eqnarray}
    {\u W}_1 - {\u W}_1\circ T_{\u \omega} &=& -S{\u W}_2 -\tilde{{\u E}}_1 - (\tilde{A}{\u \sigma})_1 \nonumber \\
    \lambda {\u W}_2 - {\u W}_2\circ T_{\u \omega} &=& - \tilde{{\u E}}_2 - (\tilde{A}{\u \sigma})_2\ .
    \label{system}
\end{eqnarray}
The system of equations \eqref{system} has an upper triangular structure and the way to solve it is summarized in Algorithm \ref{algorithm-newton}
described in Appendix~\ref{apend:newton-algorithm}.

Implementing the Newton method to find approximations of ${\u K}$,
${\u \mu}$, and computing suitable norms of such approximations
will allow us to get information on the breakdown threshold of
invariant attractors (see Section \ref{sec:Sobolev}).  {The
algorithm to implement one step of the Newton method, thoroughly
explained in \cite{CallejaCdlL2013} (see Appendix
\ref{apend:newton-algorithm} for the sake of completeness), is
very efficient, since all steps involve only diagonal operations
in the Fourier space and/or diagonal operations in the real space.
Moreover, if we represent a function in discrete points or in
Fourier space, we can compute the other functions through a FFT;
using $N$ Fourier modes to discretize the function, then we need
only $O(N)$ storage and $O(N \log N)$ arithmetic operations.}

\subsection{Sobolev criterion for the 4D case}\label{sec:Sobolev}
Given a periodic function $f$ on $\torus^2$, we consider a sample
of points on a regular grid of size $\underline{L} = (L_1, L_2)$ as
\begin{equation}\nonumber
    \u\psi_j := (\psi_{j_1}, \psi_{j_2}) = \left( \frac{2\pi j_1}{L_1} , \frac{2\pi j_2}{L_2} \right),
\end{equation}
where $\underline{j}=(j_1, j_2)\in\integer^2$ and $0 \leq j_1 <
L_1$, $0\leq j_2 < L_2$. The total number of points is given by
$L_D = L_1L_2$. To implement numerically the Newton step described
in Section~\ref{sec:newton4d}, we consider the Fourier series
{expansion of} $f(\u\psi) =  \sum_{\underline{k}'}
\hat{f}_{\underline{k}'} e^{i \underline{k}'\cdot {\u\psi}}$ with
the multi-index $\underline{k}'=(k'_1, k'_2)$ given as
follows:\begin{equation}\nonumber
    k'_l = \left\{\begin{array}{cl}
         k_l & \mbox{if } 0\leq k_l \leq L_l/2  \\
       k_l-L_l  & \mbox{if } L_l/2 < k_l < L_l
    \end{array} \right. .
\end{equation}
Note that the truncated Fourier series coincides with the Discrete Fourier Transform on the points of the grid. Following \cite{Haro2016},
we also introduce the tail in the
 {$l$-th angle $\psi_l$ with $l=1,2$} of the truncated Fourier series as follows:\begin{equation}\nonumber
    \mbox{tail}_l(\{\hat{f}_{\underline k}\}):= \sum_{\mathcal{C}_l } |\hat{f}_{\underline k}|, \qquad l=1,2\ ,
\end{equation}
where $\mathcal{C}_l$ is defined as the set of multi-indices
\begin{equation*}
    \mathcal{C}_l = \left\{{\underline k}=(k_1, k_2):\ \frac{L_l}{4} \leq k_l \leq \frac{3L_l}{4}\right\}.
\end{equation*}
To control the quality of the approximation we ask this tail to be small.

The rigorous results described in \cite{CallejaCdlL2013} and
\cite{calleja2009-thesis} provide an algorithm to approximate the
analyticity breakdown. Given a well-behaved approximate solution
to \eqref{eqP}, there are true solutions in a neighborhood and,
moreover, close to the breakdown the Sobolev semi-norms  {of high
enough order} must blow up. For a function $f:\torus^2 \rightarrow
\real$,  {we introduce the $L^2$-norm as} $\|f\|_{L^2}:=
(\sum_{\u   k\in \integer^2} |\hat f_{\u k}|^2)^{1/2} $. Then,
following \cite{Bla-Lla-13}, we define the semi-norms
$$
    \|f\|_{r, ||} := \|(2\pi \underline{\omega}\cdot \nabla)^r f\|_{L^2} \ ,\qquad
    \|f\|_{r, \perp} := \|(2\pi \underline{\omega}^\perp \cdot \nabla)^r f\|_{L^2}\ ,
$$
where  {$\u\omega= (\omega_1, \omega_2)$ and $\u\omega^\perp =
(-\omega_2,\omega_1)$. These norms provide information about the
regularity of the torus on a single direction ($\u\omega$ or
$\u\omega^\perp$). The blow up of any of the norms we consider
implies the blow up of the Sobolev norm $\|\nabla^r f\|_{L^2}$. We
have introduced these semi-norms, because it is easier to compute
them numerically and because they also provide information of the
directions in which the derivatives blow up faster.} For
trigonometric polynomials, say $f^{(N)}(\u\psi) =
\sum_{|\underline{k}|\leq N} \hat{f}_{\underline{k}}
e^{i\underline{k}\cdot \underline{\psi}}$, the semi-norms are
computed as follows:\begin{equation*}
    \|f^{(N)}\|_{r,||} = \left( \sum_{|\underline{k}|\leq N} (2\pi \underline{\omega}\cdot \underline{k})^{2r}|\hat{f}_{\underline{k}}|^2\right)^{1/2}, \qquad
    \|f^{(N)}\|_{r,\perp} = \left( \sum_{|\underline{k}|\leq N} (2 \pi \underline{\omega}^\perp \cdot \underline{k})^{2r}|\hat{f}_{\underline{k}}|^2\right)^{1/2}.
\end{equation*}
The domain of existence of invariant tori can be computed using an
approximate solution $\underline{K}_\varepsilon$,
$\underline{\mu}_\varepsilon$ of equation \eqref{inv} with
$\underline{K}_\varepsilon = (K_1, K_2, K_3, K_4)$ represented by
trigonometric polynomials; a regular behavior of the Sobolev norms
of $\underline{K}_\varepsilon$, as the parameter $\ep$ increases, provides
evidence of the existence of the quasi-periodic orbit. The
algorithm to identify the boundary of the existence domain can be
described as follows:
\begin{algorithm} \label{algo:continuation}Use $\underline{K}_0$, $\underline{\mu}_0$ for the integrable case ($\ep =0$).
\begin{itemize}
    \item[] \textbf{Repeat}
    \item[] \begin{itemize}
        \item[] Increase the parameter $\ep$ along the positive real axis
        \item[] Run the Newton step (Algorithm \ref{algorithm-newton} in Appendix~\ref{apend:newton-algorithm})
        \item[] \textbf{If} iteration of the Newton step does not converge
        \begin{itemize}
            \item[] decrease the increment in the parameter $\ep$
        \end{itemize}
        \item[ ]\textbf{Else} (Iteration success)
        \begin{itemize}
            \item[] Record the values of the parameters and compute the semi-norms of the solution
            \item[]\textbf{If} for any index $l$, any of the tails $\mbox{tail}_l(\{\hat{K}_{k,i}\}) $ exceeds a given value
            \begin{itemize}
                \item [] Double the number of Fourier coefficients in the $l$-angle
            \end{itemize}
        \end{itemize}
    \end{itemize}
    \item[] \textbf{Until} one of the Sobolev seminorms of the approximate solution exceeds a given value
\end{itemize}
\end{algorithm}

\begin{remark}
We note that Algorithm  \ref{algo:continuation}  is also used for
the 2D case. The only change is that there is only one tail for
the hull function that solves \eqref{eqP}. Also, in
this case the Sobolev norm is defined  as $\|u\|_r:=
\|\partial_\theta^r u\|_{L^2}  $, and for a trigonometric
polynomial $u^{(N)}(\theta) = \sum_{|k|\leq N} \hat{u}_k
e^{ik\theta} $ the norm is computed as $\|u^{(N)}\|_r =
\left(\sum_{|k|\leq N} k^{2r} |\hat{u}_k|^2 \right)^{1/2}$.
\end{remark}

\subsection{Results for the 2D case}\label{sec:newres2d}
The results of the implementation of the Newton method and the
Sobolev criterion for the 2D case are summarized in Figure
\ref{fig:comparison-2d-1harm}, which shows the Sobolev norm  {of
the hull function $u$ appearing in  \equ{u2D}}, providing an
estimate of the breakdown threshold which is consistent with the
value 0.758 which was found in Figure~\ref{fig:1os2dim} in the
case $\omega=s^{-1}$, $\lambda=0.8$,  {$V(x)=\sin x$ (left plot)}.
The estimate of the threshold is consistent with  {0.189} for the
two harmonics case  {$V(x)=\sin (x) + \frac{1}{3} \sin (3x)$
(right plot)}. Additional details for different values of
$\varepsilon$ on the norm of the hull function and the norm of the
error are given in Table~\ref{tab:tab4}.

\begin{figure}[ht]
    \centering
   \includegraphics[trim = 8cm 10cm 8cm 10cm, width=3.4truecm]{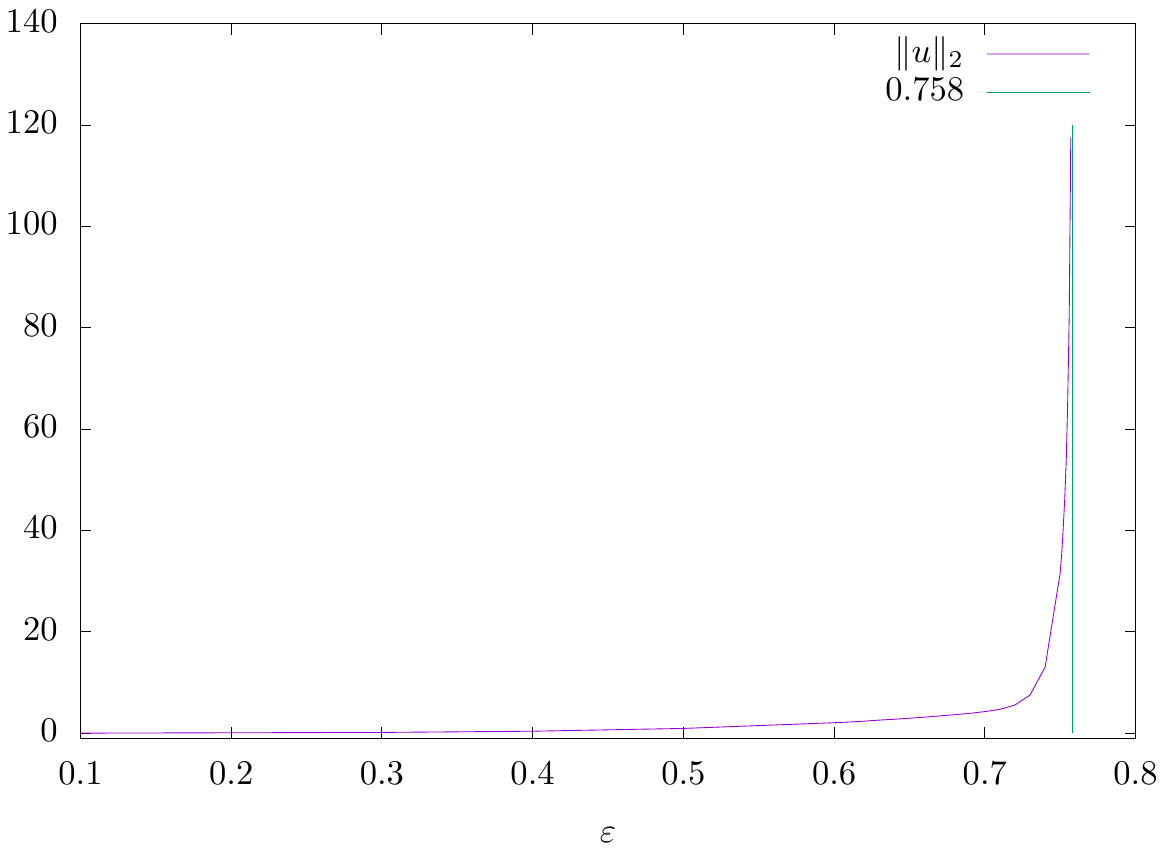}\hfil\hfil\hfil
    \includegraphics[trim = 8cm 10cm 8cm 9cm, width=3.4truecm]{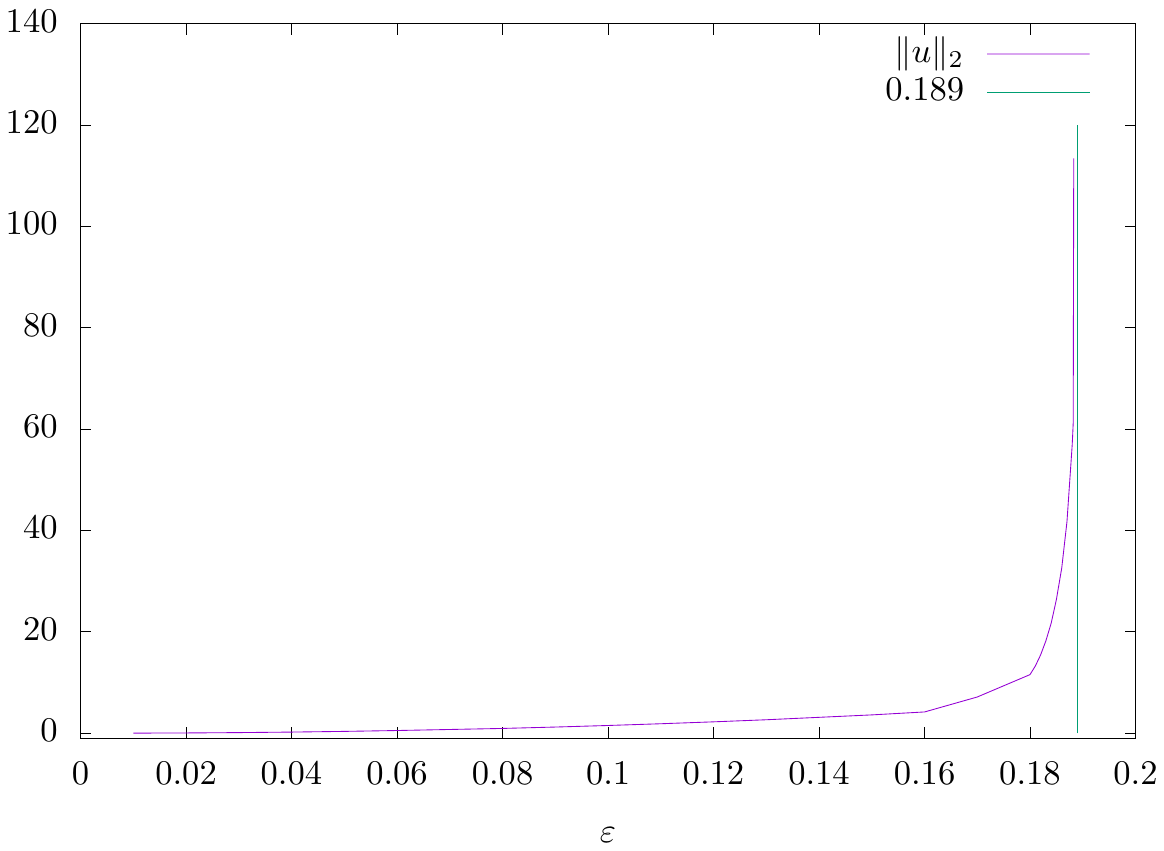}
    \caption{Map \equ{DSM2} with $\omega=s^{-1}$, $\lambda = 0.8$,  {graph of $\|u\|_2$.
Left: $V(x)=\sin x$. Right: $V_{2h}(x)=\sin (x) + \frac{1}{3} \sin (3x)$.}}
\label{fig:comparison-2d-1harm}
\end{figure}

\begin{table*}[h]
  \centering
\caption{Sobolev norms of the map \equ{DSM2} for different values of the continuation using two different potentials.
Parameters: $\omega =s^{-1}$, $\lambda = 0.8$. We have only included the values of $\ep$ for which the number  of Fourier modes were doubled. $L$ denotes the number of Fourier modes. }
\label{tab:gam}

\begin{tabular}{cc}
\hglue-1cm $V(x) = \sin (x)$  & $V(x) = \sin (x) + \frac{1}{3}\sin(3x) $\\

\begin{tabular}{|c|c|c|c|}
  \hline
  $\varepsilon$ & $\|u\|_2 $ & $\| E \|_\infty$ & $L$ \\
  \hline
0.1 & 0.040  & 8.7e-10  & 64 \\
0.3 & 0.169 & 8.0e-15  &  128\\
0.5 & 0.968 & 5.91e-12  & 256\\
0.6 & 2.082 & 8.0e-10  & 512\\
0.62 & 2.405 & 2.4e-13 & 1024 \\
0.68 & 3.664 & 1.4e-12 & 2048 \\
0.72 & 5.566 & 1.0e-10 & 4096 \\
0.75 & 31.638 & 1.6e-10 & 8192 \\
0.751 & 35.510 & 9.1e-13 & 16384 \\
0.754 & 54.069 & 4.3e-12 & 32768 \\
0.756 & 82.410 & 5.5e-11 & 65536 \\
0.757 & 117.512 & 4.3e-10 & 131072 \\
\hline
\end{tabular} &

\begin{tabular}{|c|c|c|c|}
  \hline
  $\varepsilon$ & $\|u\|_2 $ & $\| E \|_\infty$ & $L$ \\
  \hline
0.01 & 0.018  & 1.2e-12  & 64 \\
0.02 & 0.063  & 2.7e-12  & 128\\
0.05 & 0.377 & 1.6e-11  & 256\\
0.09 & 1.235 & 8.1e-11  & 512 \\
0.13 & 2.671 & 7.2e-11 & 1024 \\
0.16 & 4.218 & 1.5e-13 & 2048 \\
0.18 & 11.59 & 1.5e-10 & 4096 \\
0.183 & 18.198 & 2.1e-13 & 8192 \\
0.186 & 32.664 & 5.2e-11 & 16384 \\
0.188 & 56.853 & 9.7e-10 & 32768 \\
0.1881 & 58.969 & 2.5e-10 & 65536 \\
0.1882 & 61.253 & 2.8e-10 & 131072 \\
0.1883 & 113.47 & 3.2e-10 & 262144 \\
\hline
\end{tabular}

\end{tabular}
\label{tab:tab4}

\end{table*}

\subsection{Results for the 4D case}\label{sec:newres4d}
We consider the conformally symplectic case with
$\lambda_1=\lambda_2=0.8$. We only monitored the growth of the
Sobolev norms for the first two components of the parameterization
$K= (K_1,K_2,K_3,K_4)$  {with $K_1(\u\psi)=\psi_1+u(\u\psi)$,
$K_2(\u\psi)=\psi_2+v(\u\psi)$ as in \equ{param} with $\u
P=(u,v)$.} The results are presented in Table
\ref{tab:Sobnorms-4d-omega-spiral} and Figure
\ref{fig:sob-norms-4d-omega-spiral}.

 {The data included in Table~\ref{tab:Sobnorms-4d-omega-spiral} suggest  an anisotropic breakdown, a phenomenon that has been observed before in \cite{Bla-Lla-13}:
the derivatives of $K$ blow-up faster in the direction of $\u\omega$, than they do in the direction of $\u\omega^\perp$.
We remark that we stopped the continuation of the Newton method at the last value reported in
Table~\ref{tab:Sobnorms-4d-omega-spiral}, since the computational time was going beyond our computer limits.
However, we believe that it is reasonable to expect that the blow-up occurs within the
vertical lines at $\varepsilon=0.64$ and $\varepsilon=0.71$ shown in Figure~\ref{fig:sob-norms-4d-omega-spiral},
left panel, which correspond to the radii of the inner and outer circles of the right panel.}

 {
\begin{table*}[h]
  \centering
\caption{ $\lambda_1=\lambda_2=0.8$,
$\gamma=0.01$, $\u\omega_s$. Values of the
Sobolev norms of the first two components of the approximate
solution $K=(K_1, K_2, K_3, K_4)$. The iteration is successful in the
Algorithm  when the error of the
approximation is $<10^{-9}$. The number of Fourier coefficients in
the $l$-angle is doubled when $\mbox{tail}_l < 10^{-13}$. $L_l$ denotes the number of Fourier modes used in the angle $l$.}
\label{tab:Sobnorms-4d-omega-spiral}
\begin{tabular}{|c|c|c|c|c|c|c|c|}
\hline
$\ep$ & $\|K_1\|_{2,||}$ & $\|K_2\|_{2,||}$ & $\|K_1\|_{2,\perp}$ & $\|K_2\|_{2,\perp}$  & $L_1$ & $L_2$ & $\|E\|_\infty$ \\
\hline
 0.05 & 0.056 & 0.453 & 0.304 & 0.083 & 32 & 32 & 3.3e-12\\
 0.10 & 0.117 & 0.907 & 0.635 & 0.168 & 64 & 64 & 5.1e-12 \\
 0.15 & 0.204 & 1.371 & 1.107 & 0.254  & 64 & 64 & 8.6e-12  \\
 0.20 & 0.355 & 1.890 & 1.921 & 0.351   & 64 & 128 & 1.6e-11 \\
 0.25 & 0.607 & 2.608 & 3.281 & 0.485  & 128 & 128 & 6.2e-11 \\
 0.30 & 0.991 & 3.815 & 5.360 & 0.709  & 128 & 128 & 1.8e-10 \\
 0.35 & 1.538 & 5.901 & 8.314 & 1.097  & 128 & 128 & 5.3e-10   \\
 0.40 & 2.278 & 9.285 & 12.311 & 1.726 & 256 & 256 & 1.1e-16 \\
 0.45 & 3.247 & 14.433 & 17.537 & 2.687  & 256 & 256 & 1.1e-15  \\
 0.50 & 4.505 & 22.041 & 24.203 & 4.142  & 256 & 256 & 1.5e-14 \\
 0.55 &  6.396 & 34.398 & 32.564 & 6.813  & 512 & 512 & 5.2e-14 \\
 0.60 & 11.866 & 66.153 & 43.077 & 15.186 & 512 & 512 & 6.5e-10 \\
 0.65  & 84.682 & 226.140 & 64.864 & 60.968 & 1024 & 1024 & 2.2e-12 \\
 0.65625 & 134.608 &  296.754 & 76.873 & 83.737 & 1024 & 2048 & 7.5e-13  \\
\hline
\end{tabular}
\end{table*}
}

\begin{figure}[h]
    \centering
    \includegraphics[trim = 8cm 10cm 8cm 10cm, width=3.4truecm]{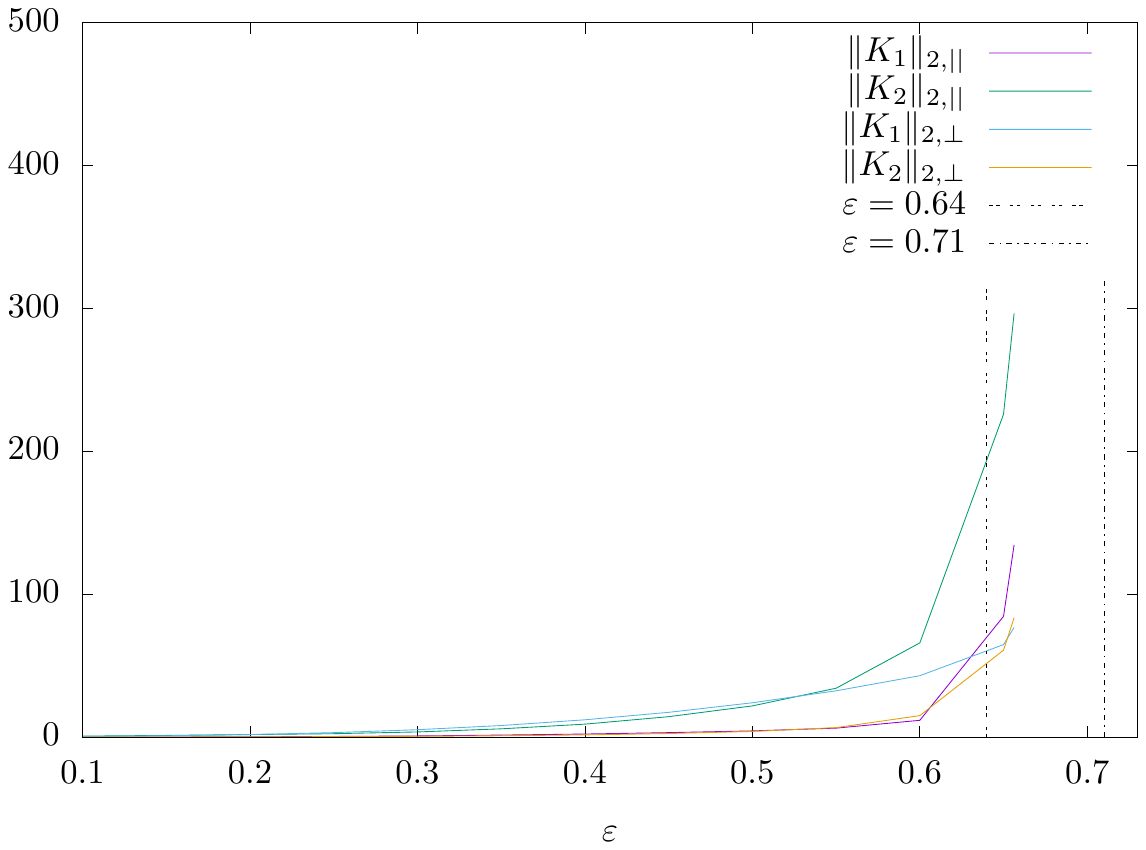}\hfil\hfil
    \includegraphics[trim = 8cm 10cm 8cm 10cm, width=3.4truecm]{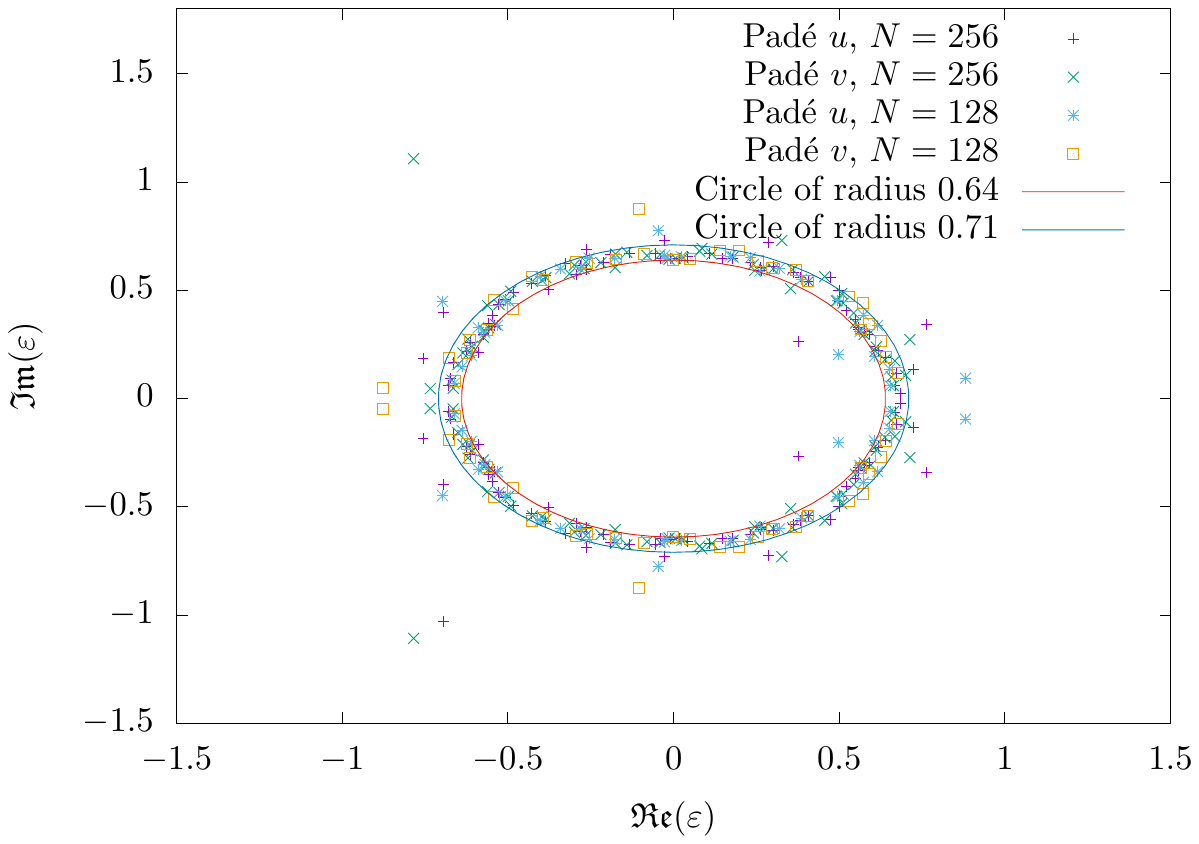}
    \caption{ {Map  {\equ{SMS} with $\u\omega_s$}, $\lambda_1=\lambda_2=0.8$,
$\gamma=0.01$.
Left panel: plots of $\ep$ vs $\|K_1\|_{2,||}$ ,
$\|K_2\|_{2,||}$, $\|K_1\|_{2,\perp}$, $\|K_2\|_{2,\perp}$, using the data in Table
\ref{tab:Sobnorms-4d-omega-spiral}. Right panel: poles of Pad\'e and log-Pad\'e
approximants using a Lindstedt series at orders $N=128$ and $N=256$. }}
\label{fig:sob-norms-4d-omega-spiral}
\end{figure}

 {\section{Approximation of the tori through periodic
orbits}}\label{sec:periodic}

 {In this Section, we compute the periodic orbits approximating
the tori with irrational frequency}. In fact, the well-known
numerical method developed by J. Greene in \cite{Greene1979} for
the determination of the stochastic transition  {of invariant
tori} is based on the conjecture that the tori breakdown when the
periodic orbits, with frequency equal to the rational approximants
to the frequency of the torus, change from stability to
instability. Greene method was originally developed for the
symplectic standard map and later used for other models, including
the 2D dissipative standard map in \cite{CallejaC10} . The
application of an extension of Greene method to dissipative
systems is made difficult by the fact that periodic orbits exist
for a whole interval of the parameters, the so-called \sl Arnold
tongues \rm (\cite{Arnold}).  {Moreover}, in the conservative case
in 2D one can take advantage of the existence of symmetry lines to
ease the search of periodic orbits; instead, symmetry lines do not
exist in the dissipative case (and they were not needed for the
computation of the Pad\'e approximants or the Sobolev norms);
 {in the 4D case, even if they exist, it is required to
implement two dimensional zero finders}. Furthermore, the
application of the extension of Greene method to higher
dimensional systems (like the 4D standard map) is not trivial,
since the search for the approximating periodic orbits might be
computationally difficult, as shown for example in \cite{CCGL20c}.

In the rest of this Section, we  {recall some results on Greene
method for dissipative systems (Section~\ref{sec:greene1}) and we}
implement a numerical method with the aim to apply an extension of
Greene method (Sections~\ref{sec:greene2d} and
\ref{sec:greene4d}).

\subsection{An extension of Greene method for dissipative systems}\label{sec:greene1}

A method for computing the breakdown threshold of quasi-periodic
solutions was presented in \cite{Greene1979} for the 2D
{symplectic} standard map, satisfying the twist condition. Greene
method is based on the assertion that the invariant circle exists
if and only if the approximating periodic orbits are at the
boundary of linear stability. A partial justification in the
symplectic case was presented in \cite{FalcoliniL92},
\cite{McKay92}, which showed that when the invariant torus exists,
one can obtain bounds on a quantity called \sl residue, \rm which
is a measure of stability of the approximating periodic orbits. {A
Greene-like method has been applied} in \cite{CallejaC10} to the
dissipative standard map, while \cite{CCFL14} extends the method
to conformally symplectic and dissipative systems in any
dimension, and provides a partial justification. The difference
with the symplectic case is that the conformally symplectic
systems need adjusting parameters and moreover they do not admit
Birkhoff invariants. The results in \cite{CCFL14} show that if the
rotational torus exists, one can predict the eigenvalues of the
approximating periodic orbits for values of the parameters close
to that of the torus; it gives also results on Arnold tongues. Two
proofs are given, the first one based on deformation theory and
the second one combining the theory of normally hyperbolic
invariant manifolds with averaging theory.  A property of
conformally symplectic systems is the \sl pairing rule, \rm which
states that the eigenvalues of a $d$-dimensional conformally
symplectic matrix with conformal factor $\lambda$ are paired;
precisely, if $\chi_i$, $i=1,...,d$, is an eigenvalue, then
$\lambda\chi_i^{-1}$ is also an eigenvalue.
The second proof of \cite{CCFL14} using the theory of normally
hyperbolic invariant manifolds applies to general systems with a
normally hyperbolic rotational circle and concludes that the
periodic orbits have $d$ eigenvalues close to 1. When the system
is conformally symplectic, using the paring rule, we recover the
same result as before. We stress that the result could apply also to the mixed systems
discussed here, but we have not pursued the implementation.

An extension of the
residue to conformally symplectic systems in any dimension can be
given in terms of the distance of the spectrum of a periodic orbit
with period $q$ to the set $\{1,\lambda^q\}$, which represents the
spectrum of the periodic orbit when the map is integrable.
According to \cite{CCFL14}, let us denote by
$$
c(x):=x^{2d}+c_{2d-1}x^{2d-1}...+c_1 x+\lambda^{dq}\ ,\ \ \ c_j\in\real\ ,
$$
the characteristic polynomial of the derivative of the
map over the full cycle of the periodic orbit with period $q$. The coefficients $c_j$ are parameterized by the
spectral numbers $\chi_j$. Consider the space $\M_0=\torus^d\times B$ with $B\subset\real^d$
a ball around zero, endowed with the standard symplectic form
$\Omega_0=\sum_{j=1}^d dy_j\wedge dx_j$,
where $\u x\in\torus^d$, $\u y\in B$; let $f_0$ be the conformal symplectic map with respect to
$\Omega_0$ defined as
$$
f_0(\u y,\u x)=(\lambda \u y,\u x+\u\omega)
$$
with $\u\omega$ Diophantine. Then, the
characteristic polynomial associated to $f_0$ is given by
$$
(x-1)^d (x-\lambda^q)^d:= x^{2d}+c_{2d-1}^0 x^{2d-1}+...+c_1^0 x+\lambda^{dq}\ ,\ \ \ c_j^0\in\real\ .
$$
Finally, the residue $R$ can be defined as the distance between
the coefficients of the characteristic polynomials:
\beq{res4D}
R:=\sum_{j=1}^d |c_j-c_j^0|\ .
\eeq
We will try to construct the approximating periodic orbits also in the dissipative 4D case with
different conformal factors to verify at least numerically whether
their linear stability provides information on the breakdown
threshold of the torus (see Section~\ref{sec:greene4d}).
 {The implementation of the method in the 4D case is the following. We fix a rational approximant to the frequency of the torus.
For each approximant and for a fixed value of $\varepsilon$, we compute a corresponding periodic orbit
within the associated Arnold tongue. In principle, we should take
the supremum of the residue over all the periodic orbits within the Arnold tongue and we
should take periodic orbits far from the boundaries of the tongue. However, this goes beyond our computational capabilities and
we are limited to select one periodic orbit as far as possible from the boundaries of the tongue.
Then, we compute the residue as a function of $\varepsilon$; when the residue is regular,
there is evidence of the stability of the periodic orbits and hence of the existence of the torus (\cite{CCFL14})}.

\subsection{ {Periodic orbit approximants in the 2D case}}\label{sec:greene2d}
We show two examples of stable periodic orbits in Figure~\ref{fig:map2d} for the 2D dissipative standard map
with $\lambda=0.8$; the left panel refers to the approximant 610/987 to $\varpi$ and the right panel
refers to the approximant 619/820 to $s^{-1}$.

\begin{figure}[hbt!]
    \centering
    \includegraphics[width=5.cm]{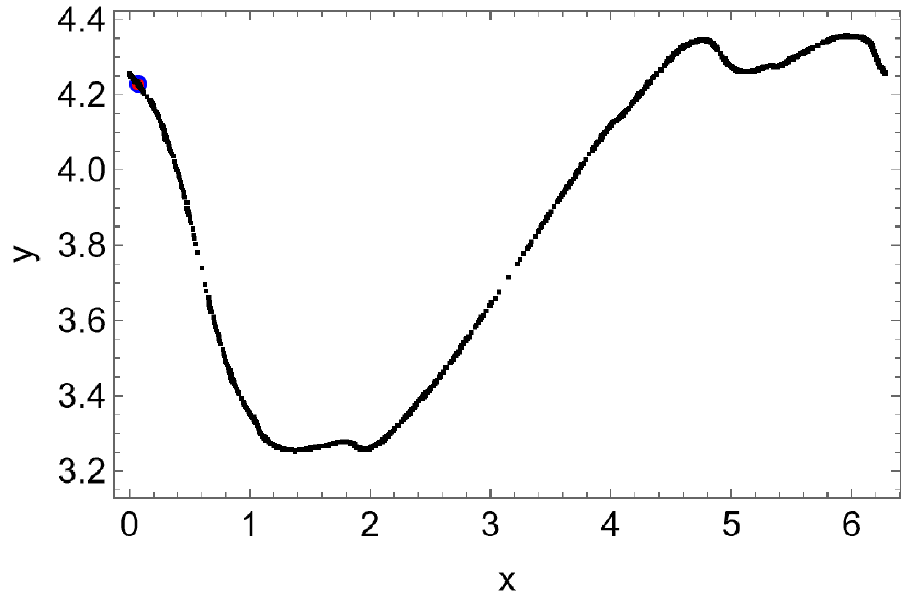}
    \includegraphics[width=5.cm]{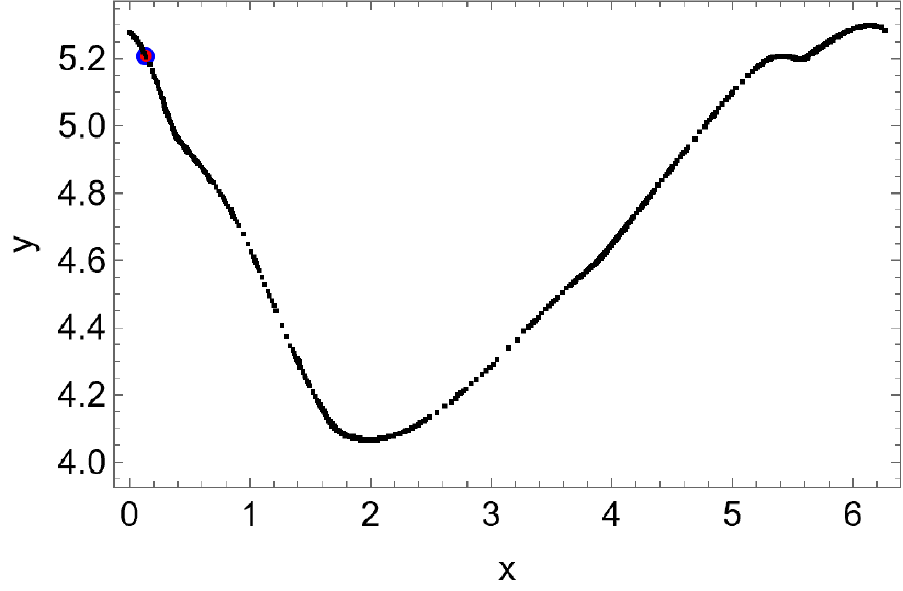}
\includegraphics[width=5.cm]{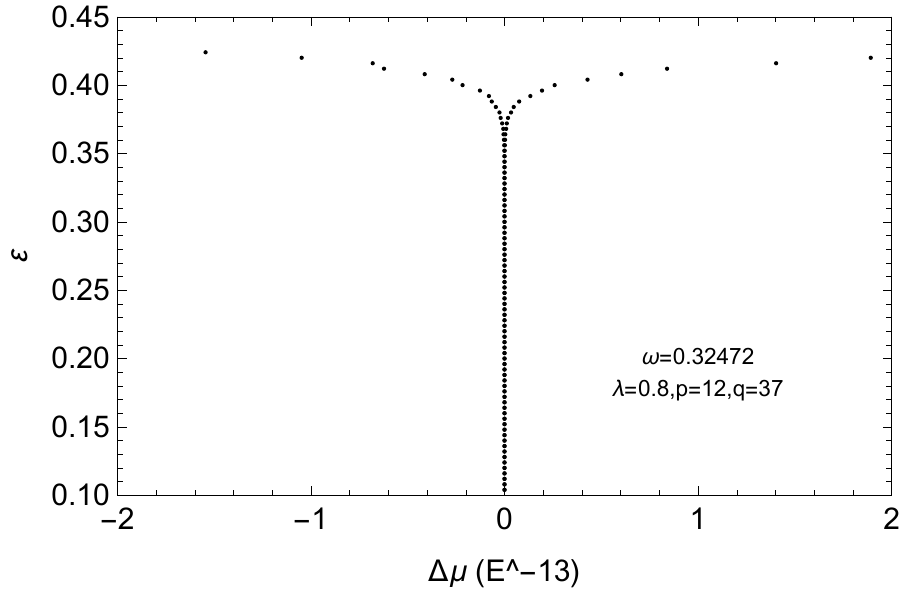}
    \caption{ {2D map \equ{DSM2} with $V(x)=\sin x$, $\lambda=0.8$. Stable periodic orbits with $\omega_1=610/987$, $\varepsilon=0.973$ (left)
and $\omega_1=619/820$, $\varepsilon=0.758$ (middle); we start the iteration of the map from the red point.}
 {In the right panel we show a numerical approximation of the Arnold tongue for $\omega=s^{-1}$, $\lambda=0.8$, $(p,q)=(12,37)$.}}
\label{fig:map2d}
\end{figure}

\begin{table*}
  \centering
\caption{ {Critical parameter and drift for stable periodic orbits
of the map \equ{DSM2} with $\lambda=0.8$, approximating to the
frequencies $\omega_1=\varpi$ (left) and $\omega_1=s^{-1}$
(right).}} \label{tab:phi2d}
\begin{tabular}{|c|c|c|c|c|c|c|c|c|}
  \hline
  $p$ & $q$ & $\varepsilon_c$  & $\mu$ && $p$ & $q$ & $\varepsilon_c$ & $\mu$\\
  \hline
2 & 3 & 1.384 & 0.82796 && 3 & 4 & 1.156 & 0.924791 \\
3 & 5 & 1.185 & 0.75105 && 37 & 49 & 0.787 & 0.935855 \\
5 & 8 & 1.123 & 0.78265 && 40 & 53 & 0.779 & 0.935884 \\
8 & 13 & 1.028 &  0.76804 && 77 & 102 & 0.774 & 0.935981 \\
13 & 21 & 1.004 & 0.77153 && 271 & 359 & 0.753 & 0.936199 \\
21 & 34 & 0.992 &  0.77043 && 619 & 820 & 0.758 & 0.936139 \\
34 & 55 & 0.984 & 0.77090 &&&&&\\
55 & 89 & 0.970 & 0.77090 &&&&&\\
89 & 44 & 0.946 & 0.77121 &&&&&\\
144 & 233 & 0.945 & 0.77119 &&&&&\\
233 & 377 & 0.975 & 0.77088 &&&&&\\
377 & 610 & 0.976 & 0.77086 &&&&&\\
610 & 987 & 0.973 &  0.77089 &&&&&\\
\hline
\end{tabular}
\end{table*}

 {
It is known that in the dissipative case the periodic orbits
can be found within an interval of the drift parameter, the Arnold tongue;
an example of a numerical approximation is given in Figure~\ref{fig:map2d}, right panel.
As shown in \cite{Arnold2} and \cite{Wenzel}, the width of the Arnold tongue scales as $\varepsilon^q$,
where $q$ is the period of the periodic orbit.}

The results of the extension of Greene method for $\varpi$
and $s^{-1}$ are presented in Table~\ref{tab:phi2d}, which shows
the drift $\mu$ and the maximum values of the critical parameter
$\varepsilon_c$ for which the periodic orbit with frequency $p/q$
given by the rational approximants to the irrational frequency is
stable. The values $\varepsilon_c=0.973$ for $\varpi$ and
$\varepsilon_c=0.758$ for $s^{-1}$ are consistent with the values
shown in Figure~\ref{fig:gold2dim} and \ref{fig:1os2dim} (upper
right plot) based on the computation of the poles of Pad\'e
approximants as well as with the results for $s^{-1}$ given in
Section~\ref{sec:newton} using Sobolev criterion.

\subsection{ {Periodic orbit approximants in the 4D case}}\label{sec:greene4d}
 {Let us denote by $\u f:\real^2\times\torus^2\rightarrow \real^2\times\torus^2$
the map \equ{SMS} with the potential in \equ{W}.}
Let $\widetilde{\u f}:\real^4\rightarrow\real^4$ be the lift of $\u f$ and let
$\u\omega_p=({p_1/q},{p_2/q})$ be a frequency vector
with $p_1,p_2\in\integer$, $q\in\integer\backslash\{0\}$. A
periodic orbit with frequency $\u\omega_p$ satisfies the condition
$$
\widetilde{\u f}^q(\u X)=\u X+(0,2\pi p_1,0,2\pi p_2)\
$$
for $\u X=(y,x,w,z)$. Searching for periodic orbits is equivalent
to finding the roots of the function \beq{eqG} \u G(\u
X):=\widetilde{\u f}^q(\u X)-\u X-(0,2\pi p_1,0,2\pi p_2)\ . \eeq
We recall that the frequency can be computed as  {in
Remark~\ref{rem:omega}. A heuristic method for finding periodic
orbits is presented in Appendix~\ref{app:po}};  {we stress that
the method is just based on experimental evidence and on an
a-posteriori corroboration that we effectively found the desired
periodic orbits. Unfortunately we are not aware of more efficient
(and possibly rigorously justified) methods to find periodic
orbits in the non-symplectic 4D case including drifts.}

The linear stability of the periodic orbits is calculated by computing the eigenvalues $\kappa_1, \kappa_2$ of the
product of the Jacobian of the map over the periodic orbit.
If $\max(|\kappa_1|,|\kappa_2|))>1$ the orbit is said to be unstable.
 {An example of the computation of the linear stability for some approximants to $\u \omega_s$ is given
in Table~\ref{tab:4dcoupl} in the dissipative case with $\lambda_1=0.8$, $\lambda_2=0.7$;
the convergence to a limit value is not evident and possibly larger approximants would be needed.}


 {The evidence of the existence of the Arnold tongues in 4D (see \cite{CCFL14}) is obtained
by taking a reference value of the drifts and initial conditions for a given periodic orbit, having fixed
the perturbing, coupling and dissipative parameters, and then by taking nearby random values of the drifts and initial conditions, which are used
as initial guess of the root finding method to find another periodic orbit of the same period. An example is given in
Figure \ref{fig:tongue} for the periodic orbit with $p_1=3$, $p_2=7$, $q=9$
(one of the first approximants to $\u\omega_s$). We
analyzed the conformally symplectic 4D map with $\lambda_1=\lambda_2=0.8$ and the coupling parameter
equal to $\gamma=0.01$. The graphs show the regions (and a zoom) of the drifts $\mu_1$, $\mu_2$ as a function of
the perturbing parameter $\varepsilon$, delimited by the blue
and red dots in which one can find a periodic orbit; as in the 2D case,
the results are consistent with the qualitative and theoretical
results given in \cite{Arnold2} and \cite{Wenzel}. Figure~\ref{fig:tongue} shows also
a 3D picture of the Arnold tongue (lower left panel).}

\begin{figure}[h]
\centering
\includegraphics[width=7cm]{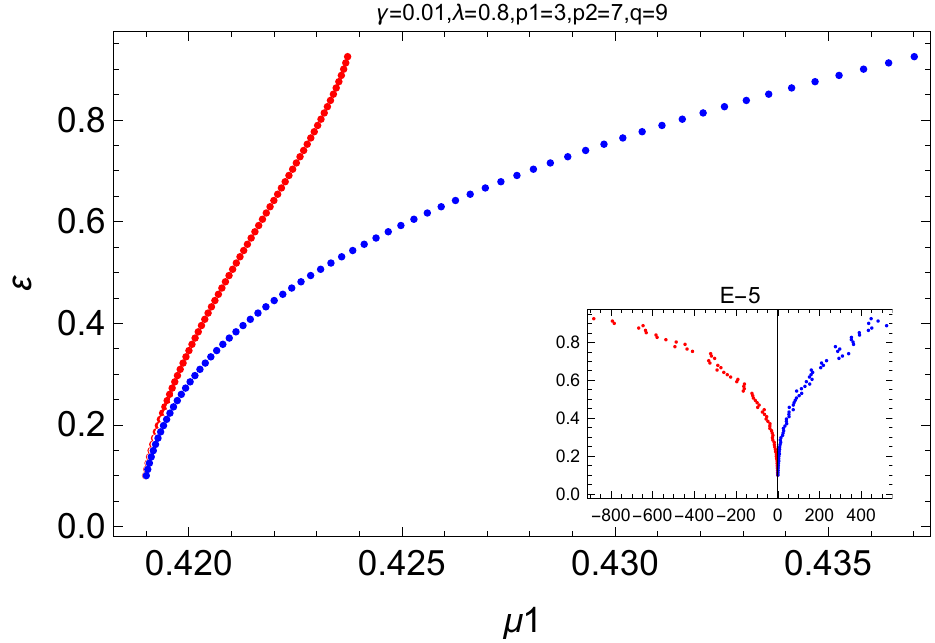}
\includegraphics[width=7cm]{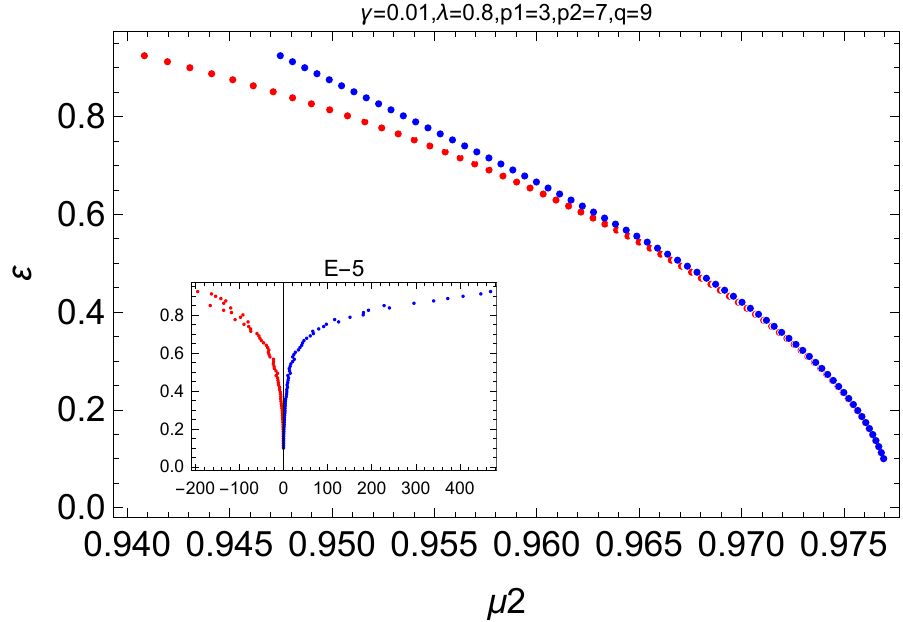}\\
\includegraphics[width=7cm]{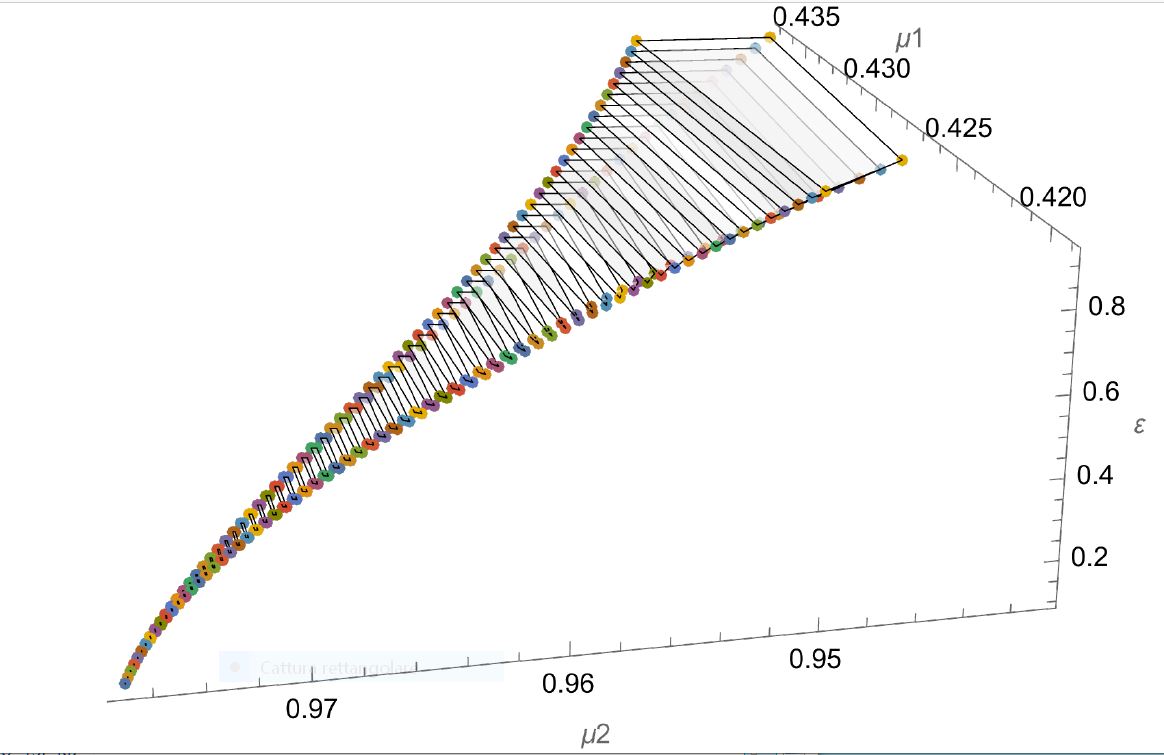}
\includegraphics[width=7cm]{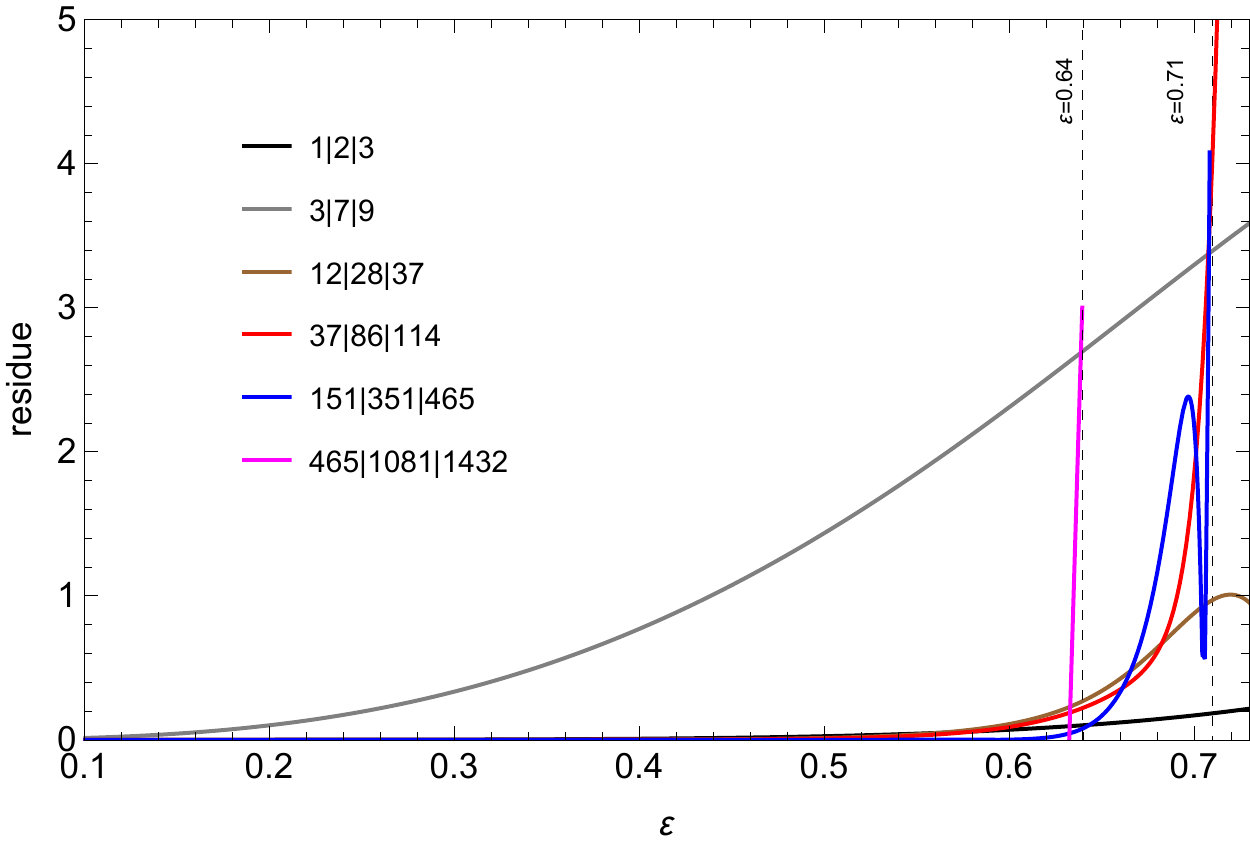}
\caption{ {Map \equ{SMS}-\equ{W} with $\lambda_1=\lambda_2=0.8$,
$\gamma=0.01$.
    Regions of the drifts $\mu_1$ (upper left), $\mu_2$ (upper right),
$\mu_1$ and $\mu_2$ (lower left) versus $\varepsilon$ for
$p_1=3$, $p_2=7$, $q=9$. The red and blue dots give the region within which one can find a
periodic orbit. Residues of the first few approximants to $\u \omega_s$ for increasing values of $\varepsilon$ (lower right).}}
\label{fig:tongue}
\end{figure}

\begin{table*}
  \centering
\caption{Parameters and initial conditions for the periodic orbits
approximating $\u \omega_s$ with frequency
$\u\omega=(p_1/q,p_2/q)$ for $\gamma=0.01$, $\lambda_1=0.8$,
$\lambda_2=0.7$.  {The values of $\varepsilon$ are those for which
the orbits are still linearly stable and just before they become
unstable.}} \label{tab:4dcoupl}
\begin{tabular}{|c|cccccccccc|}
\hline
& $p_1$ & $p_2$ & $q$ & $\varepsilon$  & $y$ & $x$ & $w$ & $z$ & $\mu_1$ & $\mu_2$ \\
\hline
 1 & 1 & 2 & 3 & 0.152 & 2.1414 & 3.0828 & 4.240 & 0.069 & 0.4191 & 0.8374
\\
\hline
 2 & 3 & 7 & 9 & 0.934 & 2.3460 & 3.0916 & 4.330 & 3.230 & 0.4238 & 0.9463
\\
 3 & 12 & 28 & 37 & 0.516 & 1.8524 & 6.2093 & 5.050 & 0.128 & 0.4097 &
0.9427 \\
 4 & 37 & 86 & 114 & 0.541 & 1.8433 & 6.2090 & 5.050 & 0.126 & 0.4098 &
0.9395 \\
 5 & 151 & 351 & 465 & 0.704 & 1.7718 & 6.2057 & 4.390 & 3.220 & 0.4078 &
0.9367 \\
 6 & 465 & 1081 & 1432 & 0.639 & 1.7992 & 6.2083 & 5.120 & 0.131 & 0.4090 &
0.9378 \\
\hline
\end{tabular}
\end{table*}

 {
We investigated further the breakdown threshold by computing the residue of the
approximating periodic orbits according to the definition given in \equ{res4D}.
Within the numerical ranges found for the Arnold tongues, we tried to select the
periodic orbits far from their boundaries.
The results for some approximants to $\u \omega_s$ are shown in the lower right panel
of Figure~\ref{fig:tongue}.}

Unfortunately, the implementation of this method
does  not allow to draw definite conclusions about the breakdown threshold,
although we notice a blow-up of the residue within the region found
to bound the domain of the Pad\'e approximants of Figure~\ref{fig:sob-norms-4d-omega-spiral}.



\begin{table*}[H]
  \centering
\caption{Parameters and initial conditions for the periodic orbits
approximating $\omega_1=p_1/q$ and $\omega_2=p_2/q$ for
$\gamma=0.01$, $\lambda_1=0.8$, $\lambda_2=0.8$. Values are
provided close to breakdown threshold in $\varepsilon$ where the
orbits are still linearly stable and just before they become
unstable (u).} \label{tab:4dcoupl}
\begin{tabular}{|c|cccccccccc|}
\hline
& $p_1$ & $p_2$ & $q$ & $\varepsilon$  & $y$ & $x$ & $w$ & $z$ & $\mu_1$ & $\mu_2$ \\
\hline
    1 & 4 & 5 & 6 & - & - & - & - & - & - & - \\
2 & 4 & 5 & 7 & 0.687 & 3.5559 & 3.1517 & 4.780 & 0.099 & 0.7175 & 0.8885 \\
3 & 5 & 6 & 8 & $>1$ & 4.3064 & 0.0812 & 4.230 & 3.220 & 0.7802 & 0.9205 \\
4 & 18 & 22 & 29 & 0.722 & 4.1431 & 0.0730 & 5.170 & 0.131 & 0.7764 & 0.9387 \\
5 & 63 & 77 & 102 & 0.726 & 4.1225 & 0.0718 & 5.160 & 0.130 & 0.7728 & 0.9365 \\
6 & 131 & 160 & 212 & - & - & - & - & - & - & - \\
7 & 411 & 502 & 665 & 0.719 & 4.1076 & 0.0669 & 4.380 & 3.230 & 0.7734 & 0.9365 \\
\hline
\end{tabular}
\end{table*}

 {\section{Conclusions}\label{sec:conclusions}}

 {
We have investigated the breakdown of rotational invariant tori in 2D and 4D standard maps, using
different formulations: symplectic, conformally symplectic, mixed, dissipative maps. We have implemented
three different methods: $(i)$ the computation of the Lindstedt series expansions, $(ii)$ the construction
of invariant tori through a Newton method, $(iii)$ the approximation of the tori through periodic orbits.}

 {
The Lindstedt series expansions to a given order allows one to estimate the radius of convergence and the domain
of analyticity of the tori through the computation of the Pad\'e approximants in the complex parameter plane.
The Newton method allows one to construct an approximation of the parameterization and the drift, whose Sobolev
norms provide an indication of the breakdown of the tori when the norms blow up. The construction of the periodic orbits
with frequencies approximating that of the invariant torus is at the basis of Greene method, according to
which the breakdown of the tori is related to the stability properties of the approximating
periodic orbits.}

 {
We stress that method $(i)$ can be implemented in all maps, method $(ii)$ in the formulation
given in \cite{CallejaCdlL2013} works only for symplectic and conformally symplectic systems,
though \cite{CH2} provides an extension to the general dissipative case with different conformal factors, method $(iii)$ has
partial rigorous justifications as shown in \cite{FalcoliniL92}, \cite{McKay92}, \cite{CCFL14}.}

 {
In the 2D cases, all three methods agree to a fairly good degree of accuracy. The domains of
analyticity are regular when the potential has only one harmonic and become flower-shaped or
irregular when the potential has 2 harmonics. Taking different values of the dissipative parameters, we noticed that
often the breakdown
thresholds are smaller as we approach the symplectic limit at which branch singularities become more
evident (\cite{Lla-Tom-94}, \cite{Lla-Tom-95}).}

 { In the 4D cases, the implementation of all methods requires
a bigger computational effort. Also in this case, we often noticed
a decrease of the threshold from dissipative to symplectic,
independently of the frequency as shown in Table~\ref{tab:omegas}.
Besides, the threshold decreases when the coupling parameter
increases, see Table~\ref{tab:gammas}. The computation of periodic
orbits in the 4D cases becomes particularly difficult, due to the
fact that one has a large set of unknowns, given by the four
initial conditions and the two drifts. In all non-symplectic
cases, one cannot rely on the existence of symmetry lines and one
needs to find the periodic orbits in their Arnold tongues, whose
structure certainly deserves a deeper investigation. Due to these
difficulties, our implementation of an extension of
Greene method in higher dimensions was inconclusive.}

\vskip.1in

\bf Acknowledgements. \rm We are grateful to  {Renato Calleja for
useful suggestions. We are deeply indebted with Rafael de la Llave
for many discussions, which helped us to improve this work. We
also thank the anonymous referee for many comments that greatly
contributed to improve the presentation of this work.}

\vskip.1in

\begin{appendix}

\section{ {Jacobi-Perron algorithm}}\label{app:JP}

For a 2D vector $\u \omega=(\omega_1,\omega_2)$, the Jacobi-Perron
algorithm (see \cite{Schweiger1990}) allows one to construct the integer sequences $p_{1,i}$,
$p_{2,i}$, $q_i$, such that $(p_{1,i}/q_i,p_{2,i}/q_i)$ are
rational approximants to $\u\omega$. For $d=2$, consider the map
$$
T(x,y)=\left(\frac{y}{x}-a_1,\frac{1}{x}-b_1\right) \ ,
$$
with
$$
a_1=a_1(x,y)=\left\lfloor\frac{y}{x}\right\rfloor \ ,
\quad b_1=b_1(x,y)=\left\lfloor\frac{1}{x}\right\rfloor \ ,
$$ and $a_j(x,y)=a_1(T^{j-1}(x,y))$, $b_j(x,y)=b_1(T^{j-1}(x,y))$.
According to \cite{Schweiger1990}, the approximants in terms of sequences
$p_{1,n}$, $p_{2,n}$, $q_n$ are defined by the matrix product
$$
\left(
\begin{array}{ccc}
 p_{1,n-2} & p_{1,n-1} & p_{1,n} \\
 p_{2,n-2} & p_{2,n-1} & p_{2,n} \\
 q_{n-2} & q_{n-1} & q_n \\
\end{array}
\right)=
\prod_{j=1}^n
\left(
\begin{array}{ccc}
 0 & 0 & 1 \\
 1 & 0 & a_j \\
 0 & 1 & b_j \\
\end{array}
\right)\ .
$$
We notice that
$$
\left(
\begin{array}{ccc}
    p_{1,n-1} & p_{1,n} & p_{1,n+1} \\
    p_{2,n-1} & p_{2,n} & p_{2,n+1} \\
    q_{n-1} & q_{n} & q_{n+1} \\
\end{array}
\right)=
\left(
\begin{array}{ccc}
    p_{1,n-2} & p_{1,n-1} & p_{1,n} \\
    p_{2,n-2} & p_{2,n-1} & p_{2,n} \\
    q_{n-2} & q_{n-1} & q_{n} \\
\end{array}
\right)
\left(
\begin{array}{ccc}
 0 & 0 & 1 \\
1 & 0 & a_{n+1} \\
0 & 1 & b_{n+1} \\
\end{array}
\right) \ .
$$
Taking only the $3rd$ column of the matrix product at the $(n+1)$-th step,
we get the vectorial recursive formula:
$$
\left(
\begin{array}{c}
p_{1,n+1} \\
p_{2,n+1} \\
q_{n+1} \\
\end{array}
\right)=
\left(
\begin{array}{c}
p_{1,n-2} \\
p_{2,n-2} \\
q_{n-2} \\
\end{array}
\right)+
\left(
\begin{array}{c}
p_{1,n-1} \\
p_{2,n-1} \\
q_{n-1} \\
\end{array}
\right)a_{n+1}+
\left(
\begin{array}{c}
p_{1,n} \\
p_{2,n} \\
q_{n} \\
\end{array}
\right)b_{n+1} \ .
$$
The recursion can be solved up to order $n>0$ using the starting values
(that leave the matrix product unchanged):
$$
\left(
\begin{array}{ccc}
    p_{1,-2} & p_{1,-1} & p_{1,0} \\
    p_{2,-2} & p_{2,-1} & p_{2,0} \\
    q_{-2} & q_{-1} & q_0 \\
\end{array}
\right) =
\left(
\begin{array}{ccc}
 1 & 0 & 0 \\
 0 & 1 & 0 \\
 0 & 0 & 1 \\
\end{array}
\right) \ .
$$

\section{Pad\'e approximants}\label{sec:Pade}
The analyticity domains of the Lindstedt series expansions of Section~\ref{sec:Lindstedt}
can be obtained through the computation of the Pad\'e approximants, that we briefly recall
as follows, referring to \cite{Bak-96} for further details. Given the series
$$
W(\ep)=\sum_{k=1}^\infty W_k \ep^k\ ,
$$
we define the Pad\'e approximants to $W$ of order $[M,N]$ for $M,N\in\nat$ as the functions
$\P_{M,N}(\ep)$, which are the ratio of the polynomials $Q_1^{(M)}(\ep)$, $Q_2^{(N)}(\ep)$ of degree, respectively, $M$ and $N$, say
$$
\P_{M,N}(\ep):= {{Q_1^{(M)}(\ep)}\over {Q_2^{(N)}(\ep)}}\ ,
$$
such that the Taylor expansion of $\P_{M,N}(\ep)$
coincides with the Taylor expansion of $W$ up
to order $M+N$:
$$
W(\ep)Q_2^{(N)}(\ep)-Q_1^{(M)}(\ep)=O(\ep^{M+N+1})\ .
$$
We can expand $Q_1^{(M)}$ and $Q_2^{(N)}$ as
$$
Q_1^{(M)}(\ep)=\sum_{j=0}^M Q_{1,j} \ep^j\ ,\qquad Q_2^{(N)}(\ep)=\sum_{j=0}^N Q_{2,j} \ep^j\ .
$$
Then, the function $\P_{M,N}(\ep)$ contains $M+N+1$ unknown coefficients, since we can
normalize $Q_{2,0}=1$.

The coefficients $Q_{1,j}$, $Q_{2,j}$ can be computed through the following recursive formulae:
\beqano
W_i+\sum_{j=1}^N W_{i-j}Q_{2,j}&=&0\ , \qquad M<i\leq M+N\ ,\nonumber\\
W_i+\sum_{j=1}^i W_{i-j}Q_{2,j}&=&Q_{1,i}\ , \qquad 0\leq i\leq M\ ,
\eeqano
where the first equation gives $Q_{2,j}$ and the second equation gives $Q_{1,j}$.

The analyticity domain of the function $W$ can then be obtained by computing the zeros of $Q_{2,N}$.
Typically, one considers diagonal Pad\'e approximants of order $[N,N]$. It is often necessary to eliminate fake zeros,
which can be distinguished from genuine ones for the fact that fake zeros disappear when the order of the
Pad\'e approximants changes or when the parameters are slightly changed (\cite{BCCF}).

\section{Algorithm of the Newton step}\label{apend:newton-algorithm}

In the algorithm below for the Newton step  {(see
\cite{CallejaCdlL2013})}, we use the following notation:
$\overline{B}$ means the average of $B$, while $(B)^o$ denotes the
zero average part, namely $(B)^o=B-\overline{B}$.

\begin{algorithm}\label{algorithm-newton}
Given ${\u K}:\torus ^2 \longrightarrow \torus^2 \times \real^2$, ${\u \mu}\in \real^2$, we perform the following computations:
\begin{itemize}
    \item[1)] ${\u E} \leftarrow {\u f}_{{\u \mu},\ep}\circ {\u K} - {\u K}\circ T_{\u \omega}$
    \item[2)] ${\alpha} \leftarrow D{\u K}$
    \item[3)] $N \leftarrow [{\alpha}^\top {\alpha} ]^{-1}$
    \item[4)] $M \leftarrow [{\alpha}, J^{-1}{\alpha} N]$
    \item[5)] $\beta \leftarrow M^{-1}\circ T_{\u \omega}$
    \item[6)] $\tilde{{\u E}} \leftarrow \beta {\u E} $
    \item[7)] $P \leftarrow {\alpha} N$ \\
    $\gamma \leftarrow {\alpha}^\top J^{-1} {\alpha}$\\
    $S \leftarrow (P\circ T_{\u \omega})^\top D{\u f}_{{\u \mu},\ep}\circ {\u K} J^{-1} P - \lambda (N\circ T_{\u \omega})^\top
    (\gamma\circ T_{\u \omega}) (N\circ T_{\u \omega})$ \\
    $\tilde{A} \leftarrow M^{-1}\circ T_{\u \omega} D_{\u \mu} {\u f}_{{\u \mu}, \ep}\circ {\u K}$
    \item[8)]  $({\u B}a)^0$ solves $\lambda ({\u B}a)^0 - ({\u B}a)^0\circ T_{\u \omega} = - (\tilde{{\u E}}_2)^0$\\
    $({ B}b)^0$ solves $\lambda ({ B}b)^0 - ({ B}b)^0\circ T_{\u \omega} = - (\tilde{A}_2)^0$
    \item[9)] Find $\overline{{\u W}}_2$, ${\u \sigma}$ solving $$ \left( \begin{array}{cc}
        \overline{S} & \overline{ S({B}b)^0} + \overline{\tilde{A}_1}   \\
        (\lambda -1)Id & \overline{\tilde{A}_2}
    \end{array}\right) \left(\begin{array}{c}
         \overline{{\u W}}_2  \\ {\u \sigma}
    \end{array}\right) = \left(\begin{array}{c}
         -\overline{\tilde{{\u E}}_1} - \overline{S({\u B}a)^0}  \\ -\overline{\tilde{{\u E}}_2}
    \end{array}\right) $$
    \item[10)] $({\u W}_2)^0  = ({\u B}a)^0 + ({B}b)^0 {\u \sigma}$
    \item[11)] ${\u W}_2 = ({\u W}_2)^0 + \overline{{\u W}}_2$
    \item[12)] $({\u W}_1)^0 $ solves $ ({\u W}_1)^0 - ({\u W}_1)^0\circ T_{\u \omega} = -(S{\u W}_2)^0- (\tilde{{\u E}}_1)^0 - (\tilde{A}_1)^0{\u \sigma} $
    \item[13)] ${\u K}\leftarrow {\u K} + M{\u W}$\\
    ${\u \mu} \leftarrow {\u \mu} + {\u \sigma}$.

    \end{itemize}

\end{algorithm}

\section{ {A method to find periodic orbits}}\label{app:po}
 {To solve \equ{eqG}, we proceed as follows. To have an initial guess, we start with the
uncoupled 4D case, which splits into a conservative and a dissipative map. We find
a solution for the two uncoupled systems for $\varepsilon$ small; in the dissipative case,
we compute a periodic orbit within the Arnold tongue. We go back
to the 4D coupled case, finding the solution for a given $\gamma$ not zero and using the solution  for
$\gamma=0$ as initial guess. We proceed to find a solution for larger values of $\varepsilon$ through
a continuation method in $\varepsilon$. A major limitation of this approach is that, within the Arnold tongue, we have not found a unique way to continue a periodic orbit.}

To be more specific, we consider  {the map \equ{SMS} with the
potential in} \equ{W} starting with $\gamma=0$, so that we obtain
two uncoupled systems;  {we first consider the case of} a
dissipative and a conservative 2D standard map. Denote by $\u
M_d:\real\times\torus\rightarrow\real\times\torus$ the  {2D
dissipative standard map in the variables $(y,x)$}. For a fixed
$\lambda$, we look for $(y_{0},x_{0},\mu)$ such that $\u M_d$
admits a periodic orbit with period $2\pi p_1/q$. From \equ{eqG}
with $\gamma=0$ we need to fulfill the condition \beq{eqGf}
(y_{0},x_{0})=\widetilde{\u M}_d^q(y_{0},x_{0};\mu)-(0,2\pi p_1)\
, \eeq where $\widetilde{\u M}_d:\real^2\rightarrow\real^2$
denotes the lift of $\u M_d$. Since we aim to find three unknowns
$(y_0,x_0,\mu)$ that solve \equ{eqGf}, the system is clearly
underdetermined,  {so we need an additional constraint on the
drift parameter.} To overcome this problem, we implemented the
following procedure, which we have heuristically found and which
turns out to be useful for our purposes.

Let us introduce $\mu'\in\real$ as $\mu' = y'(1-\lambda)-\varepsilon\sin(x')$, which is
not constant as $(y',x')$ vary. 
We define the mapping
$\widetilde{\u M}_e:\real^3\rightarrow\real^3$ as an extension of $\u M_d$ as follows:
\beqano
y'&=& \lambda y+\mu+\varepsilon \sin(x)\nonumber\\
x'&=& x+y' \nonumber\\
\mu'&=& y'(1-\lambda)-\varepsilon\sin(x') \ .
\eeqano
 {We are able to find periodic orbits of the extended map, if $(y_0, x_0, \mu_0)$ fulfill the following equation
\beq{eqGfp}
(y_{0},x_{0},\mu_0)=\widetilde{\u M}_e^q(y_{0},x_{0},\mu_0)-(0,2\pi p_1,0)\ .
\eeq
Since any triple $(y_{0},x_{0},\mu_0)$ that fulfills \eqref{eqGfp} also fulfills \eqref{eqGf}, with $\mu=\mu_0$, we found  a $2\pi p_1/q$ periodic orbit of the 2D dissipative map. We notice that the solution of equation \eqref{eqGfp} will generate two distinct different orbits for the original and extended mappings. Only for $\mu=\mu_0$ we will find the right periodic orbit for the original dissipative map. Since the method is heuristic, we further validate the result by checking that the frequency associated to the solution
$(y_{0},x_{0},\mu_0)$ of \eqref{eqGf} coincides with the first component of $\u\omega$.}

Let us denote the
 {2D conservative map in the variables $(w,z)$} as $\u M_c:\real\times\torus\rightarrow \real\times\torus$.
To get a periodic orbit with frequency equal to $2\pi p_2/q$, we need to solve
\beq{eqGg}
(w_{0},z_{0})=\widetilde{\u M}_c^q(w_{0},z_{0})-(0,2\pi p_2)\ ,
\eeq
where we have introduced the lift $\widetilde{\u M}_c:\real^2\rightarrow\real^2$.
The solution of \equ{eqGfp} and \equ{eqGg} provides an initial guess for \eqref{SMS} with $\gamma=0$.


In the 4D case with $\gamma\neq 0$, we proceed as follows. We assume to have a
starting point given by $\u X_0:= (y_{0},x_{0},w_{0},z_{0})$,
so that for $\gamma=0$ we obtain a periodic orbit for $\u M_d$ with period {\bf $p_1/q$}
starting from
$(y_{0},x_{0})$ and a periodic orbit for $\u M_c$ with period $p_2/q$ starting
from $(w_{0},z_{0})$.
We apply\footnote{This procedure is computationally
    expensive, meaning that: $(i)$ the root finding method might need a large number of iterations to converge to the requested accuracy;
    $(ii)$ the computations require accuracy, which means even more than 1000 digits in the iteration of the mapping;
    $(iii)$ if the root finding fails, one needs to find better initial conditions which is highly non trivial  {and it
        requires intervention};
    $(iv)$  the continuation method only works with a small step for the perturbing parameter.}
 {a continuation method in $\varepsilon$}. 
However, since the solution
$\u X^{(\gamma)}=(y^{(\gamma)},x^{(\gamma)},w^{(\gamma)},z^{(\gamma)})$ for $\gamma\neq0$ requires to adapt
$(y_{0},x_{0},\mu_0)$, as well as $(w_{0},z_{0})$ simultaneously at each step of the continuation method,
we are led to solve the system of equations
\beqa{eqGp}
y^{(\gamma)}= \tilde f_{1}^q(\u X^{(\gamma)},\mu^{(\gamma)}) &\quad&
x^{(\gamma)}= \tilde f_{2}^q(\u X^{(\gamma)},\mu^{(\gamma)}) - 2\pi p_1 \nonumber \\
w^{(\gamma)}= \tilde f_{3}^q(\u X^{(\gamma)},\mu^{(\gamma)}) &\quad&
z^{(\gamma)}= \tilde f_{4}^q(\u X^{(\gamma)},\mu^{(\gamma)}) - 2\pi p_2 \nonumber \\
\mu^{(\gamma)}= \tilde f_{e}^q(\u X^{(\gamma)},\mu^{(\gamma)}) &&
\eeqa
with respect to $\u X^{(\gamma)}$ and $\mu^{(\gamma)}$, having defined
$$
\tilde f_{e} := y(1 - \lambda) - \varepsilon\frac{\partial}{\partial x}W(x,z;\gamma) \ .
$$
Let us now consider the generalization to the dissipative 4D map
with $0<\lambda_{1,2}\leq1$, $\mu_{1,2}\in\mathbb R$ and $W$,
$\varepsilon$, $\gamma$ as in  {\equ{SMS}}. Using the same
approach outlined for the mixed case, we introduce the extension
of  {the dissipative 4D map} through the following additional
mapping equations:
$$
\mu_1'=
y'(1 - \lambda_1) - \varepsilon \frac{\partial}{\partial x}W(x',z';\gamma)\ ,\qquad
\mu_2'=
w'(1-\lambda_2) - \varepsilon \frac{\partial}{\partial z}W(x',z';\gamma) \ .
$$
Any periodic orbit of \equ{SMS}  {with frequency $(p_1/q,p_2/q)$}
needs to fulfill the conditions
$y_{q}^{(\gamma)}=y_{0}^{(\gamma)}$,
$x_{q}^{(\gamma)}=x_{0}^{(\gamma)}+2\pi p_1$,
$w_{q}^{(\gamma)}=w_{0}^{(\gamma)}$,
$z_{q}^{(\gamma)}=z_{0}^{(\gamma)}+2\pi p_2$ with
$\mu_{1,q}^{(\gamma)}=\mu_{1,0}^{(\gamma)}$,
$\mu_{2,q}^{(\gamma)}=\mu_{2,0}^{(\gamma)}$, where the subindex
$q$ means the $q$-th iterate. Hence, we are led to solve the
system of equations \equ{eqGp}, replacing the last line with
\beq{eqGDDp} \mu_1^{(\gamma)}= \tilde f_{e_1}^q(\u
X^{(\gamma)},\mu_1^{(\gamma)}) \ ,\qquad \mu_2^{(\gamma)}= \tilde
f_{e_2}^q(\u X^{(\gamma)},\mu_2^{(\gamma)}) \eeq with $(\mu_1,
\mu_2) =(\mu_1^{(\gamma)},\mu_2^{(\gamma)})$, having defined the
functions
$$
\tilde f_{e_1} :=
y(1 - \lambda_1) - \varepsilon \frac{\partial}{\partial x}W(x,z;\gamma)\ ,\qquad
\tilde f_{e_2} :=
w(1 - \lambda_2) - \varepsilon\frac{\partial}{\partial z}W(x,z;\gamma) \ .
$$
 {We notice that the drift parameters are dynamical variables
in the extended mapping. However, their initial values also solve the conditions for periodic orbits of the
original mapping (removing the last line in \equ{eqGp}). Thus, the initial values for the drift parameters
can also be taken as fixed parameters in the original mapping and provide the correct periodic orbit
also in the constant drift system.}

\end{appendix}

\bibliographystyle{abbrv}
\bibliography{mixedref}

\end{document}